\documentclass
[
]
{scrartcl}

\usepackage[l2tabu, orthodox]{nag}
\usepackage{etoolbox}
\newtoggle{final}
\togglefalse{final}
\toggletrue{final}

\usepackage{xcolor}
\usepackage[pagebackref,pdfusetitle]{hyperref}
\iftoggle{final}{%
\colorlet{linkcolour}{blue!40!black}
\colorlet{urlcolour}{magenta!30!black}
}{%
\colorlet{linkcolour}{blue!50!black}
\colorlet{urlcolour}{magenta}
}
\hypersetup{
pdfauthor={Robert I McLachlan, Klas Modin, Olivier Verdier, Matt Wilkins},
colorlinks=true,
linkcolor=linkcolour, citecolor=linkcolour,
urlcolor=urlcolour,
}

\usepackage[numbers]{natbib}
\usepackage[blocks,auth-lg,affil-it]{authblk}
\usepackage{url}

\usepackage{standalone}
\usepackage{pgfplots}
\usepackage{subfig}
\usepackage{booktabs}
\usepackage{tikz}
\usetikzlibrary{positioning,calc,decorations.pathmorphing,intersections,decorations.markings,spy}

\usepackage{datetime}

\yyyymmdddate
\newcommand*{\printdate}{\today}

\iftoggle{final}{}{%
\usepackage[top]{background}
\SetBgContents{\textsf{Draft \printdate{} {\tiny\currenttime}}}
\SetBgScale{2}
\SetBgOpacity{.5}
}

\usepackage{ifxetex}
\ifxetex
\usepackage{fontspec}
\defaultfontfeatures{Ligatures=TeX} 
\usepackage{newunicodechar}
\newunicodechar{α}{\alpha}
\newunicodechar{β}{\beta}
\newunicodechar{δ}{\delta}
\newunicodechar{ε}{\varepsilon}
\newunicodechar{ζ}{\zeta}
\newunicodechar{θ}{\theta}
\newunicodechar{λ}{\lambda}
\newunicodechar{μ}{\mu}
\newunicodechar{ν}{\nu}
\newunicodechar{ξ}{\xi}
\newunicodechar{π}{\pi}
\newunicodechar{η}{\eta}
\newunicodechar{ω}{\omega}
\newunicodechar{Φ}{\Phi}
\newunicodechar{Λ}{\Lambda}
\newunicodechar{≤}{\leq}
\newunicodechar{≥}{\geq}
\newunicodechar{∞}{\infty}
\newunicodechar{≠}{\neq}
\newunicodechar{∂}{\partial}
\newunicodechar{Δ}{\Delta}
\newunicodechar{∆}{\triangle}
\newunicodechar{∫}{\int}
\newunicodechar{≈}{\approx}
\newunicodechar{φ}{\varphi}
\newunicodechar{κ}{\kappa}
\newunicodechar{ψ}{\psi}
\newunicodechar{∑}{\sum}
\newunicodechar{Γ}{\Gamma}
\else
\usepackage[utf8x]{inputenc}
\SetUnicodeOption{mathletters}
\DeclareUnicodeCharacter{952}{\theta}
\fi

\newcommand*\Poisson[2]{\{#1,#2\}}
\newcommand*\bracket[2]{\left\langle #1,#2\right\rangle}
\newcommand*\Bracket[2]{\Big\langle #1,#2\Big\rangle}

\newcommand*\dd{\mathrm{d}}

\newcommand*\sympman{\mathcal{P}}
\newcommand*\Man{\mathcal{M}}
\newcommand*\ManQ{\mathcal{Q}}
\newcommand*\Mp{\Man^{H}}
\newcommand*\Mpd{\widetilde{\Man}^{H}}
\newcommand*\Tan{\mathsf{T}}

\newcommand*\ee{\mathrm{e}}
\newcommand*\ii{\mathrm{i}}
\newcommand*\Span{\operatorname{span}}

\newcommand*\R{\mathbf{R}}
\newcommand*\RR{\mathbf{R}}
\newcommand*\ZZ{\mathbf{Z}}
\newcommand*\CC{\mathbf{C}}
\newcommand*\cons{V}
\newcommand*\ncons{m}

\newcommand*\ddp[2]{\frac{∂ #1}{∂ #2}}

\newcommand*\qb{\overline{q}}
\newcommand*\pb{\overline{p}}

\newcommand*\inv{^{-1}}

\newcommand*\qv{\mathbf{q}}
\newcommand*\pv{\mathbf{p}}
\newcommand*\gv{\mathbf{g}}
\newcommand*\Hr{\mathcal{H}}

\usepackage{xspace}
\usepackage{relsize}
\newcommand*\Shake{{\smaller SHAKE}\xspace}
\newcommand*\Rattle{{\smaller RATTLE}\xspace}
\newcommand*\DAE{{\smaller DAE}\xspace}
\newcommand*\ODE{{\smaller ODE}\xspace}

\newcommand*\fib[1]{[#1]}

\newcommand*\Orbit{\mathcal{O}}

\newcommand*\sorth{^{\perp}}
\newcommand*\trans{^{\mathsf{T}}}
\newcommand*\TMperp[1][]{\Tan_{#1}\Man\sorth}

\newcommand*\mom{\mathsf{J}}
\newcommand*\launch[1]{Γ_{#1}}

\newcommand*\cfunc[2][]{\bracket{\mom\if\relax\detokenize{#1}\relax\else(#1)\fi}{#2}}
\newcommand*\Ham[1]{{X}_{#1}}

\newcommand*\codim{\operatorname{codim}}

\newcommand*\proj{\pi}
\newcommand*\Proj{\Pi}
\newcommand*\Mq{[\Man]}

\newcommand*\omegaP{\omega}
\newcommand*\omegaM{\nu}

\newcommand*\omegaMp{\varpi}
\newcommand*\omegaMpd{\tilde{\omegaMp}}
\newcommand*\Hkx{H_{\fib{x}}}
\newcommand*\HM{H_\Man}
\newcommand*\incl[1]{\imath_{#1}}
\newcommand*\inclM{\incl{\Man}}
\newcommand*\inclMp{\incl{\Mp}}

\newcommand*\interior[2]{\mathrm{i}_{#1}#2}
\newcommand*\Smap[1]{S_{#1}}
\newcommand*\Sh{\Smap{h}}
\newcommand*\Rmap[1]{R_{#1}}
\newcommand*\Rh{\Rmap{h}}
\newcommand*\dflow[1][h]{\varphi_{#1}}
\newcommand*\dflowh{\dflow}

\tikzset{
  >=stealth,
  border/.style={draw},
  fiber/.style={draw,thick,blue!70!black},
  mzero/.style={draw,black,thick},
  mminus/.style={draw,blue,dotted},
  mplus/.style={draw,black,thick,dashed},
  flight/.style={draw,black,thick, densely dotted, decoration={markings, mark=at position 0.6 with {\arrow{>}}}, postaction={decorate}},
  helpnode/.style={shape=coordinate},
  launching/.style={fill=black},
  crashing/.style={fill=blue},
  fiberslide/.style={fiber,ultra thick,},
  rattle/.style={color=green!50!black},
  rattleslide/.style={draw,fiber,ultra thick,},
  proj/.style={draw,dashed},
  quotientmanifold/.style={draw,thick},
  zplus/.style={draw=black},
}

\title{Geometric Generalisations of \Shake and \Rattle}

\author[1]{Robert I McLachlan\thanks{\url{r.mclachlan@massey.ac.nz}}}
\author[2]{Klas Modin\thanks{\url{klas.modin@chalmers.se}}}
\author[3]{Olivier Verdier\thanks{\url{olivier.verdier@math.ntnu.no}}}
\author[1]{Matt Wilkins\thanks{\url{m.c.wilkins@massey.ac.nz}}}

\affil[1]{
	Institute of Fundamental Sciences,
	Massey University,
	Private Bag 11 222, Palmerston North 4442, New Zealand
}

\affil[2]{
	Department of Mathematical Sciences, 
	Chalmers University of Technology, 
	SE--412 96 G\"oteborg, 
	Sweden
}

\affil[3]{
	Department of Mathematical Sciences,
	NTNU,
	7491 Trondheim,
	Norway
}

\usepackage[notref,notcite,
\iftoggle{final}{final,}{}
]{showkeys}
\usepackage{rotating}

\usepackage[
textsize=footnotesize,
backgroundcolor=yellow!20!white,
bordercolor=gray,
linecolor=gray!20!white,
\iftoggle{final}{disable,}{}
]{todonotes}
\makeatletter\let\chapter\@undefined\makeatother

\usepackage{aliascnt}
\usepackage{framed}

\usepackage{mathtools}
\mathtoolsset{centercolon,showonlyrefs=false}
\newcommand{\num}[1]{#1} 

\usepackage{amsthm}
\newcommand{\newtheoremalias}[2]{%
\newaliascnt{#1}{theorem}%
\newtheorem{#1}[#1]{#2}%
\aliascntresetthe{#1}%
}

\theoremstyle{plain}
\newtheorem{theorem}{Theorem}[section]

\newtheoremalias{proposition}{Proposition}

\newtheoremalias{corollary}{Corollary}

\newtheoremalias{lemma}{Lemma}

\theoremstyle{definition}
\newtheoremalias{definition}{Definition}

\newtheorem{assumption}{Assumption}

\theoremstyle{definition}
\newtheoremalias{example}{Example}

\newtheoremalias{remark}{Remark}

\usepackage{enumitem}

\begin{document}


\maketitle

\begin{abstract}
	A geometric analysis of the \Shake and \Rattle methods for constrained Hamiltonian problems is carried out.
	The study reveals the underlying differential geometric foundation of the two methods, and the exact relation between them.
	In addition, the geometric insight naturally generalises \Shake and \Rattle to allow for a strictly larger class of constrained Hamiltonian systems than in the classical setting.
	
	In order for \Shake and \Rattle to be well defined, two basic assumptions are needed. First, a nondegeneracy assumption, which is a condition on the Hamiltonian, i.e., on the dynamics of the system. Second, a coisotropy assumption, which is a condition on the geometry of the constrained phase space.
	Non-trivial examples of systems fulfilling, and failing to fulfill, these assumptions are given. 
\end{abstract}

\begin{description}
	\item[Keywords] Symplectic integrators · Constrained Hamiltonian systems · Coisotropic
submanifolds · Differential algebraic equations
\item[Mathematics Subject Classification (2010)] 37M15 · 65P10 · 70H45 · 65L80
\end{description}

\newpage
\listoftodos

\newpage

\tableofcontents

\section{Introduction}\label{sec:intro}

\Shake and \Rattle are two commonly used numerical integration methods for Hamiltonian problems subject to holonomic constraints~\cite{RyCiBe1977,An1983,LeSk1994,Ja1996,Re96}.
The difference between the two methods is that \Rattle preserves ``hidden'' constraints, whereas \Shake does not.
For details and a historical account, see the monographs~\cite[\S\!~7.2]{LeRe2004} and~\cite[\S\!~VII.1.4]{HaLuWa06}.

In this paper we give a rigorous geometric analysis of the \Shake and \Rattle methods. 
Our approach is based on the observation by \citet{Re96} that \Shake and \Rattle may be expressed using flow maps.
The analysis sheds light on the underlying ``geometric foundation'' of the two methods, and also on the exact relation between them.
In addition, the  geometric insight allows us to integrate a larger class of constrained problems than before.
Indeed, the geometric versions of \Shake and \Rattle work for \emph{coisotropic constraints}.
This class of constraints is strictly larger than the class of holonomic constraints.
In particular, they may depend on both position and momentum.

Throughout the paper we utilise the language of differential geometry.
The main reason for doing so is \emph{not} to generalise \Shake and \Rattle from~$\R^{2d}$ to manifolds, but rather because this notation makes the geometric structures more transparent.
However, in order to link to the standard literature on \Shake and \Rattle, we give many key results also in the classical $\R^{2d}$~setting as examples.

Our notation closely follows that of \citet{MaRa99}.
In particular, if~$\Man$ and~$\mathcal{N}$ are two manifolds and $f\in\mathcal{C}^{\infty}(\Man,\mathcal{N})$, then $\Tan f\colon \Tan\Man\to\Tan\mathcal{N}$ denotes the tangent derivative.
If $\mathcal{N}\subset\Man$ is a submanifold, then $\incl{\mathcal{N}}\colon \mathcal{N}\to\Man$ denotes the inclusion, and the pull-back of differential forms from $\Man$ to $\mathcal{N}$ is denoted $\incl{\mathcal{N}}^*$.
If $(\sympman,\omegaP)$ is a symplectic manifold, and $H\in\mathcal{C}^{\infty}(\sympman)$, then $\Ham{H}$ denotes the corresponding Hamiltonian vector field, and the Poisson bracket is denoted~$\Poisson{\cdot}{\cdot}$.
The contraction between a vector~$X$ and a differential form $\alpha$ is denoted $\interior{X}{\alpha}$.

We continue this section with an outline of the paper and the main results.

\paragraph{Problem formulation} 
\label{par:problem_class}
%
Let $(\sympman,\omegaP)$ be a symplectic manifold, $H\in\mathcal{C}^\infty(\sympman)$ a smooth function, and $\Man\subset \sympman$ a submanifold.
Given $(\sympman,\omegaP,H,\Man)$, the problem is to find a smooth curve $t\mapsto\gamma(t)$ such that
\begin{subequations}\label{eq:constrained_system}
\begin{equation}\label{eq:constrained_system_extrinsic}
	\inclM^*\big(\interior{\dot\gamma}{\omegaP} - \dd H\big) = 0 , \qquad \gamma(t)\in \Man.
\end{equation}

Equation~\eqref{eq:constrained_system_extrinsic} looks like a Hamiltonian system on $\sympman$, but constrained to stay on the submanifold~$\Man$, called the \emph{constraint manifold}.
It can be rewritten as
\begin{equation}\label{eq:intrinsic_maineq}
	\interior{\dot\gamma}{\omegaM} - \dd \HM = 0,
\end{equation}
\end{subequations}
where $\omegaM = \inclM^*\omegaP$ and $\HM = \inclM^* H$.
From this formulation it is clear that equation~\eqref{eq:constrained_system} is intrinsic to~$\Man$, i.e., it only depends on objects defined on~$\Man$.


\begin{example}\label{ex:classical_setting_eq}
	Let $\sympman = \RR^{2d}$ with canonical coordinates $z=(q,p)$, and let~$\Man = g^{-1}(\{ 0 \})$, where $g=(g_1,\ldots,g_\ncons)\in\mathcal{C}^\infty(\RR^{2d},\RR^m)$.
	%
	Equation~\eqref{eq:constrained_system} then takes the form
	\begin{equation}\label{eq:constrained_system_coordinates_long}
	\begin{split}
		\dot{q} &= H_p(q,p) + g_p(q,p)\trans λ \\
		\dot{p} &= -H_q(q,p) - g_q(q,p) \trans λ \\
		0 &= g(q,p)
		,
	\end{split}
	\end{equation}
	where  $g_q$ and $g_p$ are the partial Jacobian matrices, and $\lambda = (\lambda_1,\ldots,\lambda_m)$ are Lagrange multipliers.
	Notice that if~$g$ does not depend on~$p$, then this reduces to a canonical Hamiltonian system with holonomic constraints.

	Owing to the Hamiltonian structure of the equations, the reduction of the differential--algebraic equation (\DAE )~\eqref{eq:constrained_system_coordinates_long} to an ordinary differential equation (\ODE ) takes a particular form.
This was already noticed by \citet[\S\!~1]{dirac2001}, and later perfected in~\cite{Go1978}.
\end{example}


\paragraph{Existence and uniqueness} 
\label{par:existence_and_uniqueness}

Since equation~\eqref{eq:constrained_system} is intrinsic to~$\Man$, it is clear that any condition or assumption for existence and uniqueness should also be intrinsic, so it is enough to investigate existence and uniqueness intrinsically on~$\Man$.

The 2--form $\omegaM$ is closed, but in general degenerate, so $(\Man,\omegaM)$ is not, in general, a symplectic manifold.
Instead, it is a \emph{presymplectic} manifold.

The kernel of $\omegaM$ defines a distribution on $\Man$ denoted $\ker\omegaM$.
As detailed in~\autoref{sub:foliation}, the kernel distribution is \emph{integrable}.
Thus, for each $z\in\Man$, there is a submanifold $\fib{z}\subset\Man$ such that $\Tan_x\fib{z} = \ker\omegaM_x$ for each $x\in\fib{z}$.

If $\gamma(t)$ is a solution to~\eqref{eq:intrinsic_maineq}, then~$\dd \HM(\gamma(t)) \in \omegaM(\Tan_{\gamma(t)}\Man)$.
Thus, solutions stay on the \emph{hidden constraint set}, given by
\begin{equation} \label{eq_def_Mp}
	\Mp = \big\{\, z\in\Man : \dd \HM(z) \in \omegaM(\Tan_z\Man) \,\big\} .
\end{equation}
In particular, a necessary condition for equation~\eqref{eq:intrinsic_maineq} to have a solution is that the initial data fulfils $\gamma(0)\in \Mp$, which is assumed from here on.

As is further explained in \autoref{sec:presymplectic_dynamics}, a sufficient condition for (local) existence and uniqueness of solutions of equation~\eqref{eq:intrinsic_maineq}, and hence equation~\eqref{eq:constrained_system}, is the following.

\begin{assumption}[Nondegeneracy] \label{assump:non_degeneracy}
	For any $z\in\Man$, the critical points of the function $H_{\fib{z}}=\incl{\fib{z}}^*\HM$ are nondegenerate.
	
	%
\end{assumption}

\begin{example}
	For the classical setting in \autoref{ex:classical_setting_eq},
	\[
		\Mp = \big\{\, z\in \RR^{2d}: g_i(z)=0,\, \Ham{H}(z)\cdot \nabla g_i(z)=0, \, i=1,\ldots,\ncons \,\big\}
	\]
	and \autoref{assump:non_degeneracy} means that the matrix~$g_z(z)\trans H_{zz}(z) g_z(z)$ is invertible for~$z\in\Mp$.
	If $g$ does not depend on $p$, then this is slightly weaker than the classical assumption that $g_q H_{pp} g_q$ is invertible (see \autoref{sec:holonomic_solvability}).
\end{example}


\paragraph{Geometric \Shake and \Rattle} 
\label{par:geometric_shake_and_rattle}

We now define \Shake and \Rattle geometrically (see \autoref{fig_shake} for an illustration of the geometrical setting).


\begin{definition}[Geometric \Shake]\label{def:geometric_shake}
	Let $\dflowh$ be a method approximating $\exp(h\Ham{H})$ for a given time step $h$.
	The \Shake mapping $\Sh\colon \Man\ni z_0 \mapsto z_1^-\in \Man$ is defined implicitly by
	\begin{equation*}
			\dflowh(y) \in \Man, \quad y \in \fib{z_0}\cap O,\quad
			z_1^- = \dflowh(y) ,
	\end{equation*}
	where $O\subset\Man$ is a suitable open subset containing~$\Mp$.
\end{definition}

\begin{definition}[Geometric \Rattle]\label{def:geometric_rattle}
	Let $\dflowh$ be a method approximating $\exp(h\Ham{H})$ for a given time step $h$.
	The \Rattle mapping $\Rh\colon \Mp\ni z_0 \mapsto z_1\in \Mp$ is defined implicitly by
	\begin{equation*}
			z_1 \in\fib{z_1^-}\cap\Mp, \quad z_1^- = \Sh(z_0).
	\end{equation*}
\end{definition}

These are abstract definitions of \Shake and \Rattle.
In order to practically be able to implement them, an implicit definition of $\Man$ in terms of constraint functions, and a parameterisation of $\fib{z_0}$, is needed.
This issue is discussed in \autoref{sec:geometry_of_shake_and_rattle}, and is related to \autoref{assump:coisotropy} introduced below.

In the holonomic case it is already known that \Shake and \Rattle essentially yield the same method, since the projection step at the end of \Rattle is ``neutralised'' by the projection step in \Shake.
This observation is made geometrically precise in \autoref{sec:geometry_of_shake_and_rattle}, where we show that \Shake and \Rattle are two different representations of the same fibre mapping.

\begin{example}
\label{ex:pendulum}
	It may be illuminating to understand the effect of both methods in a familiar example, that of a planar pendulum, realised as a constrained mechanical system.
	The ambient space is $\RR^4$, and the constraint, which is holonomic, is given by $g(\qv,\pv) \coloneqq \|\qv\|^1 - 1$.
	The constraint manifold $\Man$ is thus a three-dimensional submanifold of $\RR^4$.
	The Hamiltonian is that of an unconstrained mass in a constant gravity field, $H(\qv,\pv) \coloneqq \frac{1}{2}\|\pv\|^2 + q_2$.
	The hidden constraint manifold is the submanifold of $\Man$ which consists of the points where the velocity is tangent to $\Man$, that is
	\begin{equation}
	\Mp = \big\{\, (\qv,\pv) \in \Man : \qv \cdot \pv = 0 \,\big\}
	.
	\end{equation}
	For this particular case, the \Shake and \Rattle algorithms are illustrated on \autoref{fig_pendulum}.
\end{example}

\begin{figure}
	\centering
	\includegraphics{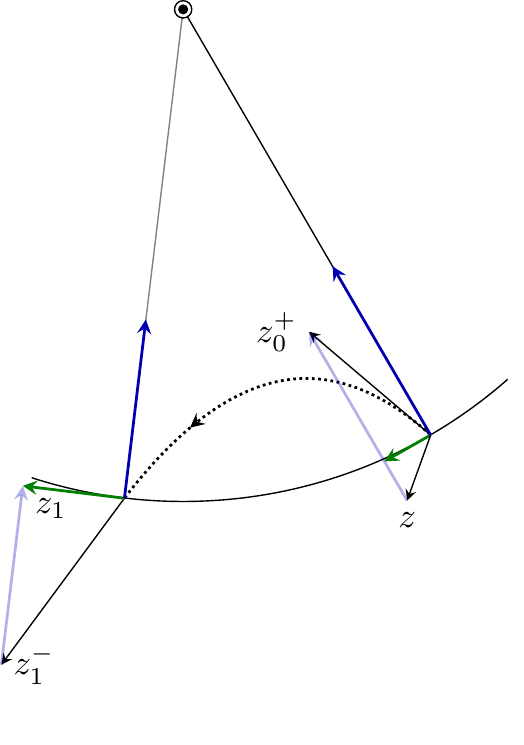}
	\caption[Pendulum]{
	An illustration of \Shake and \Rattle in the familiar case of the planar pendulum.
	The reader is encouraged to compare that situation with the general geometrical description on \autoref{fig_shake}.

	We start from a point $z$, which is on $\Man$, but not necessarily on $\Mp$, which means that the velocity is not assumed to be tangent to the constraint circle.
	The arrow in the direction of the rod represents a kick applied to the mass, which has the effect of changing its normal velocity, resulting in a point $z_0^+$ in the phase space.

	The next step is to use a symplectic integrator to simulate the unconstrained flow, until the mass reaches the constraint manifold, that is, until the mass is at distance one from the origin; in this case we use the exact, unconstrained solution, which is a parabola.
	One has to adjust the initial kick in order to reach the constraint manifold exactly after one time-step.

	The mass reaches the constraint manifold $\Man$ at the phase point $z_1^-$ with a non-tangential velocity, which means that $z_1^-$ does not belong to $\Mp$.
	The effect of \Rattle is now simply to correct the normal velocity by applying an appropriate kick in the rod direction in order to obtain a point $z_1 \in \Mp$.
	}
	\label{fig_pendulum}
\end{figure}


\paragraph{Well-posedness of \Shake and \Rattle} 
\label{par:well_posedness_of_shake_and_rattle}

The algebraic equations defining \Shake and \Rattle can be thought of as discretisations of the original equation~\eqref{eq:constrained_system}.
However, contrary to the continuous case, the discretised equations are \emph{not} intrinsic to~$\Man$.
Thus, well-posedness of \Shake and \Rattle depends on how~$\Man$ is embedded in~$\sympman$.

Let $\Tan\Man^\bot$ denote the orthogonal complement of $\Tan\Man$ with respect to the symplectic form~$\omegaP$, i.e., $u\in\Tan_x\Man^\bot$ if and only if $\omegaP(u,v) =0$ for all $v\in\Tan_x\Man$.

\begin{definition}\label{def:coisotropy}
	A submanifold $\Man$ of $\sympman$ is called \emph{coisotropic} if $\Tan\Man^\bot\subset \Tan\Man$.
\end{definition}
As is explained carefully in \autoref{sec:extrinsic_viewpoint}, the natural assumption in order for \Shake and \Rattle to be well-posed is the following, which is a completely extrinsic condition, i.e., it only has to do with how~$\Man$ is embedded in~$\sympman$.

\begin{assumption}[Coisotropy] \label{assump:coisotropy}
	$\Man$ is a coisotropic submanifold of $\sympman$.
\end{assumption}

\begin{example}
	For the setting in \autoref{ex:classical_setting_eq}, let $g_1,\ldots,g_\ncons$ be the components of the vector valued constraint function~$g$.
	Then \autoref{assump:coisotropy} means that
	\[
		\Poisson{g_i}{g_j}(z) = 0 , \quad \forall\, z\in\Man .
	\]
	An equivalent interpretation of \autoref{assump:coisotropy} is that none of the Lagrange multipliers in equation~\eqref{eq:constrained_system_coordinates_long} are resolved by differentiating the constraint condition once.
	From a \DAE point of view, the nondegeneracy and coisotropy assumptions together asserts that equation~\eqref{eq:constrained_system_coordinates_long} has index~3.
	An important particular case is obtained when $g$ does not depend on $p$.
	In that case, \autoref{assump:coisotropy} \emph{always holds} (see \autoref{sec:holonomic_case}).
	As shown in \autoref{sub:coisotropic_embedding}, \autoref{assump:coisotropy} also implies that $\fib{z}$ is parameterised by 
	\[
		(\lambda_1,\ldots,\lambda_\ncons)\longmapsto \exp\Big(\sum_{i=1}^\ncons \lambda_i X_{g_i}\Big)(z).
	\]
	In turn, this means that the geometric \Shake method is given by
	\begin{equation*}
		\Sh = \dflowh\circ\exp\Big(\sum_{i=1}^m \lambda_i X_{g_i}\Big)
	\end{equation*}
	where $\lambda_1,\ldots,\lambda_m$ are determined implicitly by the conditions~$g_i \circ \Sh = 0$.
	Likewise, the geometric \Rattle method is given by
	\begin{equation*}
		\Rh = \exp\Big(\sum_{i=1}^m \sigma_i X_{g_i}\Big)\circ \Sh
	\end{equation*}
	where $\sigma_1,\ldots,\sigma_m$ are determined implicitly by $(X_H \cdot \nabla g_i)\circ \Rh = 0$.
\end{example}

It is easy to find instances where the coisotropic and/or the nondegeneracy assumptions do not hold, and where the \Shake and \Rattle methods are not well-defined. 
For example, if we take as constraint $H = \text{const}$, then the nondegeneracy assumption does not hold, and it is easy to see that \Shake and \Rattle are not well-defined.
This is expected, since the result by \citet*{GeMa1988} asserts that it is not (in general) possible to construct symplectic and energy preserving methods. 
In \autoref{sec_example_high_index} we give further examples of failing assumptions.
Lastly, in \autoref{sec_example_hopf} we also give a numerical example of a Hamiltonian problem with mixed position and momentum constraints, where we use the geometric \Shake and \Rattle methods.


\paragraph{Main results} 
\label{par:main_results}

The main results in the paper can be summarised as follows.

\begin{enumerate}
	\item Under \autoref{assump:non_degeneracy}, the set $\Mp$ is a symplectic submanifold with symplectic form $\omegaMp=\inclMp^*\omegaM$, and equation~\eqref{eq:constrained_system} is well-posed for initial data in $\Mp$. 
	(\autoref{thm:MH_manifold})
	
	\item Under \autoref{assump:non_degeneracy} and \autoref{assump:coisotropy}, there exists an open set $O\subset\Man$, containing $\Mp$, such that the \Shake map $\Sh\colon O\to O$ is well defined and presymplectic, i.e., $\Sh^*\omegaM = \omegaM$.
	Further, it is convergent of order at least~1.
	(\autoref{thm:Mvariational_methods} and \autoref{pro:convergence})
	
	\item Under \autoref{assump:non_degeneracy} and \autoref{assump:coisotropy}, the \Rattle map $\Rh\colon \Mp\to\Mp$ is well defined and symplectic, i.e., $\Rh^*\omegaMp = \omegaMp$.
	Further, it is convergent of order at least~1.
	(\autoref{thm:Mvariational_methods} and \autoref{pro:convergence})
	
\end{enumerate}

\section{Hamiltonian Systems on Presymplectic Manifolds} 
\label{sec:presymplectic_dynamics}

In this section we investigate the geometric structures of equation~\eqref{eq:constrained_system} from the intrinsic viewpoint, i.e., without ``looking outside'' of~$\Man$.

In general, a \emph{presymplectic manifold} is a pair $(\Man,\omegaM)$, where $\Man$ is a smooth manifold, and $\omegaM$ is a closed 2--form on~$\Man$ called a \emph{presymplectic form}. 
The difference from a symplectic form is that $\omegaM$ need not be nondegenerate.
Thus, a symplectic manifold is a special case of a presymplectic manifold.
We review some geometric concepts of presymplectic manifolds that are essential in the remainder.
For a more thorough treatment, we refer to the book by \citet*{LiMa87}.

Given a function $\HM\in\mathcal{C}^\infty(\Man)$, equation~\eqref{eq:intrinsic_maineq} constitutes a Hamiltonian system on~$(\Man,\omegaM)$.
Since $\omegaM$ might be degenerate, this equation is not, in general, an ordinary differential equation, but instead a \DAE . 
We show in \autoref{sub:nondegeneracy_assumption} that under \autoref{assump:non_degeneracy} it is an index~1 \DAE on~$\Man$.
(In \autoref{sec:extrinsic_viewpoint} we take the complementary extrinsic viewpoint, and we show that under \autoref{assump:non_degeneracy} and \autoref{assump:coisotropy} equation~\eqref{eq:constrained_system} can be interpreted as an index~3 problem on~$\sympman$.)

\subsection{Foliation}\label{sub:foliation}

Throughout the paper we make the following ``blanket assumption'':
\begin{framed}
	\centering
	The dimension of the kernel distribution~$\ker\omegaM$ is constant.
\end{framed}
One important consequence is that the distribution $\ker\omegaM$ (now assumed to be \emph{regular}) is \emph{integrable} (cf.~\cite[Th.~25.2]{GuSt90}). 
That is, at each point $x\in\Man$ there is a submanifold $\fib{x}\subset \Man$ passing through~$x$ whose tangent spaces coincides with the distribution. 
The submanifolds $\fib{x}$ are called \emph{leaves}, and the collection $\Mq$ of all leaves is called a \emph{foliation}.
See \autoref{fig_shake} for an illustration of the foliation $\Mq$ of $\Man$.

\begin{remark}
	The foliation defines an equivalence class by $y\equiv x$ if $y\in\fib{x}$.
	We denote the set of all such equivalence classes by~$\Mq$. 
	The projection is given by $$\proj \colon \Man \ni x\to \fib{x}\in \Mq.$$
	The set $\Mq$ may or may not be a smooth manifold.
	When it is, the presymplectic form~$\omegaM$ descends to a symplectic form $\bar\omegaM$ on $\Mq$, and $\proj$ is a symplectic submersion.
	The projection map~$\proj$ being a submersion means that we have a \emph{fibration} of $\Man$.
	Locally, every foliation is a fibration, but not necessarily globally.
	Throughout the remainder of the paper, we use the word ``fibre'' instead of ``leaf'', although the fibration may only be local.
\end{remark}

A \emph{presymplectic mapping} is a mapping $\varphi \colon \Man\to\Man$ that preserves the presymplectic form $\omegaM$, i.e., for which
\[
	\omegaM(u,v) = \varphi^*\omegaM(u,v) \coloneqq \omegaM(\Tan\varphi \cdot u,\Tan\varphi \cdot v) \qquad
	\forall\; u,v\in \Tan_x\Man
	.
\]

There is a certain class of mappings that are trivial in the sense that they reduce to the identity mapping in the quotient manifold $\Mq$.

\begin{definition}
\label{def_trivially_presymplectic}
	A smooth mapping $φ \colon \Man\to\Man$ is called \emph{trivially presymplectic} if it preserves each fibre, i.e., if
	\begin{equation*}
	φ(x) \in \fib{x} \quad \text{for all $x \in \Man$} 
	.
	\end{equation*}
\end{definition}

The following result is clear.

\begin{proposition}\label{pro:trivially_presymplectic_maps}
	If $\varphi$ is trivially presymplectic, then it is presymplectic.
\end{proposition}

\subsection{Hidden Constraints} 
\label{sub:hidden_constraints}

The fact that $\omegaM$ does not have full rank reflects that, in general, the possible solutions of \eqref{eq:intrinsic_maineq} do not fill the whole manifold $\Man$.
Indeed, if a curve~$\gamma(t)$ is a solution of~\eqref{eq:intrinsic_maineq}, then $\dd H(\gamma(t))$ must be in the set $\omegaM(\Tan_{\gamma(t)}\Man)$.
As already seen in~\autoref{sec:intro}, the set of points at which this is fulfilled defines the hidden constraint set~$\Mp\subset\Man$, given by~\eqref{eq_def_Mp}.

\begin{remark}
In general, this set is defined as the locus of $\ncons$ functions, where $\ncons$ is the dimension of $\fib{x}$.
However, if the differential of those functions are not independent at the locus points, the set $\Mp$ need not be a submanifold, and if it is a submanifold, it need not be of codimension $\ncons$.
For instance, we may have $\Mp  = \Man$ if  the Hamiltonian~$H$ is constant along each fibre~$\fib{x}$.
This is in particular the case if $\omegaM$ is nondegenerate.
\end{remark}

\begin{remark}
The subset $\Mp$ is, strictly speaking, not a set of \emph{hidden} constraints, but rather \emph{implicit} constraints as a consequence of \eqref{eq:intrinsic_maineq}.
\end{remark}




\subsection{Nondegeneracy Assumption} 
\label{sub:nondegeneracy_assumption}

In this section we show that the nondegeneracy assumption, \autoref{assump:non_degeneracy}, ensures that: (i)~$\Mp$ is a submanifold of $\Man$; (ii)~$\omegaMp=\inclMp^*\omegaM$ is a symplectic form; and (iii)~the initial value problem~\eqref{eq:intrinsic_maineq} is a Hamiltonian problem on~$(\Mp,\omegaMp)$.
As a consequence, problem~\eqref{eq:intrinsic_maineq} has unique solutions for initial data in~$\Mp$.

From a \DAE point of view, \autoref{assump:non_degeneracy} ensures that the \DAE~\eqref{eq:intrinsic_maineq} on $\Man$ has index~1.
As it turns out (see \autoref{sec:geometry_of_shake_and_rattle} below), the nondegeneracy assumption, together with \autoref{assump:coisotropy}, also asserts that the geometrically defined \Shake and \Rattle methods are well defined.

We start with the observation that $\Mp$ is the set of critical points of~$\Hkx$.

\begin{proposition}
\label{prop_Mp_critical}
For $x\in\Man$, let $\Hkx = \incl{\fib{x}}^*\HM$. 
Then
\begin{equation*}
\Mp = \big\{\, y \in \Man : \dd \Hkx(y) = 0 \,\big\}
\end{equation*}
\end{proposition}
\begin{proof}
If $y\in\Mp\cap\fib{x}$ then $\dd\Hkx(y) = 0$ since $\dd H(y) \in \omegaM(\Tan_y\Man)$.
Thus, the set $\Mp\cap\fib{x}$ consists of critical points of the function $\Hkx$. 
\end{proof}



\begin{theorem}\label{thm:MH_manifold}
	Under \autoref{assump:non_degeneracy}, the following holds.
	\begin{enumerate}
		\item The set $\Mp$ is a submanifold of~$\Man$.
		\item 
		\label{item_span}
		At a point $x\in\Mp$ we have
		\begin{equation*}
		\Tan_x \Man = \Tan_x \Mp \oplus \ker ν
		.
		\end{equation*}
		In particular, the presymplectic form $\omegaM$ restricted to $\Mp$ is a symplectic form.
		Thus, $\Mp$ is a symplectic manifold.
		\item Equation~\eqref{eq:intrinsic_maineq} has unique solutions for initial data in~$\Mp$. 
		These solutions are given by the solutions of the Hamiltonian problem on
		$\Mp$ obtained by restricting~$H$ to~$\Mp$.
	\end{enumerate}
\end{theorem}

\begin{proof}
Each statement is proved, respectively, as follows.
\begin{enumerate}
\item
	Let $X_1,\ldots,X_\ncons$ be linearly independent vector fields on $M$ that span the distribution $\ker\omegaM$.
	Define the functions $\rho_i(x) \coloneqq \bracket{\dd H(x)}{X_i(x)}$. 
	Then, using \autoref{prop_Mp_critical}, $$\Mp = \big\{\, x\in\Man: \rho_i(x)=0, \quad i=1,\ldots,\ncons \,\big\}.$$
	$\Mp$ is a submanifold if $\dd \rho_1(x),\ldots,\dd\rho_k(x)$ are linearly independent for every~$x\in\Mp$.
	An equivalent conditions is that the matrix
	\[
		m_{ij} \coloneqq \bracket{\dd\rho_i(x)}{X_j(x)},\quad i,j=1,\ldots,\ncons
	\]
	be invertible for every~$x\in\Mp$. 
	Using that $\rho_i(x)=0$, we get in local coordinates $x_j$ that
	\[
		m_{ij} = X_i^\alpha(x) \frac{\partial^2 H(x)}{\partial x^\alpha\partial x^\beta} X_j^\beta(x)
	\]
	where $X_i = X_i^\alpha \frac{\partial}{\partial x^\alpha}$.
	Since $X_1(x),\ldots,X_k(x)$ is a linearly independent basis of $(\ker\omegaM)_x=\Tan_x\fib{x}$, \autoref{assump:non_degeneracy} means exactly that this matrix is invertible, which thus proves the first assertion.
	
\item
	For the second assertion, using that the codimension of $\Mp$ is $\ncons$, it suffices to prove that $(\ker\omegaM)_x\cap\Tan_x\Mp =  0$ for every $x\in\Mp$.
	Let $u\in\Tan_x\Mp$.
	Then $\bracket{\dd\rho_i(x)}{u} = 0$.
	Next, assume that $U\in(\ker\omegaM)_x$.
	Then~$U$ can be expanded as $U = u^i X_i(x)$.
	We now get
	\[
		0 = \sum_{i=1}^k u^i \bracket{\dd\rho_j}{X_i(x)} = m_{ij}u^i .
	\]
	Under \autoref{assump:non_degeneracy} we know that $m_{ij}$ is invertible, which implies that $U=0$.
	Thus, $(\ker\omegaM)_x\cap\Tan_x\Mp =  0$, which proves that $\nu$ restricted to $\Mp$ is nondegenerate.
	
\item
	For the final assertion, it is enough to show that $\gamma(t)$ is a solution to equation~\eqref{eq:intrinsic_maineq} if and only if it is a solution to the Hamiltonian problem
	\[
		\nu(\dot\gamma(t),U) = \bracket{\dd H(\gamma(t))}{U}, \quad\forall\, U\in\Tan\Mp
	\]
	on the symplectic manifold~$\Mp$ (for which existence and uniqueness follows from standard \ODE theory).
	As we have seen, under \autoref{assump:non_degeneracy} every $X\in\Tan_{\gamma(t)}\Man$ can be written as $X=U+W$ with $U\in\Tan\Mp$ and $W\in\ker\omegaM$.
	Now,
	\[
		\omegaM(\dot\gamma(t),X) = \omegaM(\dot\gamma(t),U)\qquad\text{and}\qquad
		\bracket{\dd H(\gamma(t))}{X} = \bracket{\dd H(\gamma(t))}{U}
	\]
	where the first and second equality follows, respectively, since $W\in\ker\omegaM$ and 
	\[
		\bracket{\dd H(\gamma(t))}{W} = 
		\Bracket{\dd H(\gamma(t))}{\sum_i w^i X_i(\gamma(t))} = 
		\sum_i w^i\rho_i(\gamma(t))=0.
	\]
\end{enumerate}
This ends the proof.
\end{proof}

\begin{remark}
If the prescribed initial condition does not lie in the set $\Mp $, there cannot be any solution curve passing through this point.
On the other hand, if $\Mp $ is a submanifold, and if it intersects the fibres of $\Mq$ cleanly, i.e., if the dimension of the intersection is constant, and if that dimension is larger than zero, then the equation may have infinitely many solutions.
This is what happens if $H$ is constant on the fibres of $\Mq$.
\end{remark}

The following result will be useful in \autoref{sec:geometry_of_shake_and_rattle}, when we analyse \Shake and \Rattle.

\begin{corollary}\label{cor:projection_onto_MH}
	Under \autoref{assump:non_degeneracy}, there exists an open set~$O\subset\Man$ containing~$\Mp$ such that the equation $y\in\fib{x}\cap\Mp$ has a unique solution for every~$x\in O$.
	The corresponding trivially presymplectic projection map $\Proj \colon O\to\Mp$, defined by $\Proj(x)=y$, is a submersion.
\end{corollary}

\begin{proof}
	This follows from \autoref{thm:MH_manifold} \autoref{item_span}, namely that for $x\in \Mp$,
\(		\Tan_x \Man = \Tan_x \Mp \oplus \ker ν\).
\end{proof}

\section{Coisotropic Constraints} 
\label{sec:extrinsic_viewpoint}

In this section we study the geometry of problem~\eqref{eq:constrained_system} from the extrinsic viewpoint.
That is, we study properties of~$\Man$ as a submanifold of the symplectic manifold~$(\sympman,\omegaP)$.
Notice that $\omegaM := \inclM^*\omegaP$ is a presymplectic form on~$\Man$, since $\dd\omegaM = \dd\inclM^*\omegaP = \inclM^*\dd\omegaP = 0$.
Thus, any submanifold of a symplectic manifold is automatically a presymplectic manifold.


\subsection{Lagrange Multipliers}

Typically, a constraint manifold is defined in terms of a number of constraint functions.
To this extent, let $\cons$ be a vector space of dimension~$\ncons$, and denote by $\cons^*$ its dual.
Let $\mom \colon \sympman\to\cons^*$ be a smooth function such that the constraint submanifold $\Man$ is given by 
\begin{equation}
\label{eqdefMG}
\Man = \mom\inv(0) = \big\{\, z\in \sympman : \mom(z) = 0 \,\big\}
.
\end{equation}
If $0$ is a regular value for $\mom$, i.e., if $\Tan \mom(z)$ has full rank for all $z \in \sympman$ such that $\mom(z) = 0$, then $\Man$ is indeed a regular submanifold of $\sympman$.
The dimension $\ncons$ of $\cons$ is the number of constraints, i.e., the codimension of $\Man$.

The problem \eqref{eq:constrained_system} may now be reformulated as finding a smooth curve
\[
	t\mapsto \big(z(t),\Lambda(t)\big)\in\sympman\times\cons
\] 
such that
\begin{subequations} \label{eqconslag}
\begin{equation}
\begin{split}
ω ( \dot{z} ) &= \dd H(z) + \dd\! \cfunc{Λ}(z) 
, \\
0 &= \mom (z)
.
\end{split}
\end{equation}
Here, the notation $\cfunc{\Lambda}$ means the smooth function $z\mapsto \bracket{\mom(z)}{\Lambda}$, depending on the parameter~$\Lambda$.
The equation can equivalently be written as
\begin{equation}
\label{eqconsdirac}
\begin{split}
\dot{z} &= \Ham{H}(z) + \Ham{\cfunc{Λ}}(z) 
 \\
0 &= \mom (z)
.
\end{split}
\end{equation}
\end{subequations}

We sometimes single out a basis $\{ e_i \}_{i = 1, \ldots , \ncons}$ of $\cons$ and define the functions $g_i$ by
\begin{equation}
\label{eq_gi}
g_i(z) \coloneqq \bracket{\mom(z)}{e_i} 
.
\end{equation}
Notice that in the case $\sympman = \RR^{2d}$, equation~\eqref{eqconslag} coincides with equation~\eqref{eq:constrained_system_coordinates_long} in \autoref{ex:classical_setting_eq} above, with $\lambda = (\lambda_1,\ldots,\lambda_\ncons)$ being the coordinate vector of~$Λ$, i.e., $Λ = ∑_{i=1}^m λ_i e_i$.

The system~\eqref{eqconslag} is again a \DAE.
Under \autoref{assump:non_degeneracy}, it follows from \autoref{thm:MH_manifold} above that this \DAE has unique solutions for initial data in~$\Mp$. 
From a \DAE point of view, \autoref{assump:non_degeneracy} asserts that system~\eqref{eqconslag} has index~3.


\subsection{Coisotropy Assumption} 
\label{sub:coisotropic_embedding}

Due to the solvability result imposed by \autoref{assump:non_degeneracy}, the Lagrange multipliers may be resolved as functions of $z$, which turns equation~\eqref{eqconslag} into
\begin{equation}\label{eq:diffeq_on_Mp}
	\dot z = X_H(z) + \sum_{i=1}^m \lambda_i(z)X_{g_i}(z) =: X(z).
\end{equation}
Notice that $X(z)$ is only defined for $z\in\Mp$ and also that $X(\Mp)\in \Tan\Mp$, so $X(z)$ defines an \ODE on the hidden constraint manifold~$\Mp$. 
From \autoref{thm:MH_manifold} it follows that its flow is symplectic.
However, the individual vector fields $\lambda_i(z)X_{g_i}(z)$ are \emph{not} Hamiltonian vector fields on $\sympman$ (assuming that $\lambda_i(z)$ is defined also outside of~$\Mp$).
In this section we present an assumption on the embedding $\inclM \colon \Man\to\sympman$ which ensures that vector fields of the form $f(z)X_{g_i}(z)$ are trivially presymplectic vector fields on $\Man$.
As we will see in \autoref{sec:geometry_of_shake_and_rattle}, this is essential in order to ensure presymplecticity and symplecticity of \Shake and \Rattle.

Recall from \autoref{def:coisotropy} that~$\Man$ is a \emph{coisotropic} submanifold of~$\sympman$ if $\TMperp \subset \Tan\Man$.
Also recall \autoref{assump:coisotropy} above (the coisotropy assumption), which states that~$\Man$ is a coisotropic submanifold.
We continue with some consequences of \autoref{assump:coisotropy}, which are later used in the geometric analysis of \Shake and \Rattle .


\begin{remark}
	It is straightforward to verify that $\Man$ being coisotropic is equivalent to $\TMperp$ being \emph{isotropic}, i.e., such that $\omegaP$ restricted to $\TMperp$ is zero.
\end{remark}

\begin{remark}
A coisotropic submanifold is such that  the symplectic form becomes as degenerate as possible (given a fixed number of constraints) when restricted on the submanifold.
More precisely, a coisotropic submanifold is such that the dimension of the distribution $\ker\omegaM$, i.e., dimension of the fibres of $\Mq$, is equal to the number of constraints~$\ncons$.
\end{remark}

%

\begin{remark}
\label{rk:coisotropic_embedding}
	From a theoretical point of view, \autoref{assump:coisotropy} is not a restriction on the presymplectic manifold $\Man$, since every presymplectic submanifold may be coisotropically embedded in a symplectic manifold~\cite{Go1982}.
\end{remark}

\begin{remark}
\label{rk:coisotropy_poisson_commute}
In practice as shown in \autoref{prop_coisotropy}, a sufficient condition for the manifold $\Man$ defined by the equations $g_i(z) = 0$ for $1≤i≤\ncons$ to be coisotropic is simply that
\begin{equation*}
\Poisson{g_i}{g_j}(z) = 0, \qquad i,j=1,\ldots,\ncons,\quad \forall\, z\in\Man
.
\end{equation*}
In particular, if the manifold $\Man$ is defined by one constraint, i.e. if $m=1$, then it is automatically a coisotropic submanifold.
\end{remark}

The following result gives alternative characterisations of coisotropic submanifolds.

\begin{proposition}
	\label{prop_coisotropy}
	Suppose that $\Man$ is a submanifold of $\sympman$.
	Then the following conditions are equivalent.
	\begin{enumerate}
		\item
			$\Man$ is a coisotropic submanifold, i.e., $\TMperp \subset \Tan \Man$.
		\item
			$\ker \omegaM = \Tan\Man\sorth$
	\end{enumerate}
	Further, if $\Man$ is defined implicitly by~\eqref{eqdefMG}, then the conditions are also equivalent to
	\begin{enumerate}
		\item[3.]
			For any $α,β\in\cons$, the functions $\cfunc{α}$ and $\cfunc{β}$ are in involution on $\Man$, i.e.,
			\begin{equation*}
				\Poisson{\cfunc{α}}{\cfunc{β}}(z) =  0 \qquad \text{for}
				\quad z \in \Man
				.
			\end{equation*}
		\item[4.]
			For any $\alpha\in \cons$, the Hamiltonian vector field $\Ham{\cfunc{α}}$ is tangent to $\Man$.
	\end{enumerate}
\end{proposition}

In order to prove this, let us start with a lemma concerning the span of the Hamiltonian vector fields $\Ham{\cfunc{α}}$.

\begin{lemma}
	\label{prop_orbit}
	Define the distribution 
	\[
	\Orbit = \Span \big\{\, \Ham{\cfunc{α}}(\Man): α \in \cons \,\big\}.
	\]
	Then $\Tan\Man\sorth = \Orbit$.
\end{lemma}

\begin{proof}
	We show that $\Orbit\sorth = \Tan\Man$, which is equivalent to the claim.
	\begin{equation*}
	X \in \Tan\Man 
	\iff 
	\bracket{\dd\cfunc{α}}{X} = 0 \quad \forall α\in\cons 
	\iff 
	ω(\Ham{\cfunc{α}}, X) = 0 \quad \forall α \in \cons
	\end{equation*}
\end{proof}

\begin{proof}[Proof of \autoref{prop_coisotropy}]
	We do it step by step.
	\begin{enumerate}
		\item[1$\leftrightarrow$2]
		In general,
		\begin{equation*}
			\ker\omegaM = \Tan\Man \cap \Tan\Man\sorth
			,
		\end{equation*}
		so $\ker\omegaM = \Tan\Man\sorth \iff \Tan\Man\sorth \subset \Tan\Man$, and that is the definition of coisotropicity of~$\Man$.
		\item[1$\leftrightarrow$3]
		First, for $x\in\Man$
		\begin{equation*}
		\Poisson{g_i}{g_j}(x) = 0 \iff ω(\Ham{g_i}(x),\Ham{g_j}(x)) = 0
		,
		\end{equation*}
		so the functions $g_i$ are in involution on $\Man$ if and only if $\Orbit$ (defined in \autoref{prop_orbit}) is isotropic, which is equivalent to~$\Man$ being coisotropic.
		\item[3$\leftrightarrow$4]
		Finally, it suffices to observe that for a point $x\in\Man$,
		\begin{equation*}
			\begin{split}
				\Ham{\cfunc{α}}(x) \in \Tan_x \Man &\iff \bracket{\dd \cfunc{β}}{\Ham{\cfunc{α}}}(x) = 0 \quad \forall β \in \cons \\
				& \iff \Poisson{\cfunc{α}}{\cfunc{β}}(x) = 0 \quad \forall β\in\cons
				.
			\end{split}
		\end{equation*}
	\end{enumerate}
\end{proof}

The following results follows directly from \autoref{prop_orbit} and \autoref{prop_coisotropy}.

\begin{corollary}\label{cor:fibre_parametrisation}
	Let $x\in\Man$. Then, under \autoref{assump:coisotropy}, the map 
	\[
		\cons \ni \alpha \mapsto \exp(\Ham{\cfunc{\alpha}})(x)\in \fib{x}
	\]
	is a local diffeomorphism.
	The fibre $\fib{x}$ is thus locally parametrised by $\cons$.
\end{corollary}

\begin{corollary}\label{cor:gi_presymplectic}
	Let $f\in\mathcal{C}^\infty(\sympman)$ and $\alpha\in \cons $.
	Under \autoref{assump:coisotropy} the vector field 
	\[
		X(z) \coloneqq f(z)\Ham{\cfunc{\alpha}}(z)
	\]	
	is tangent to $\Man$, and presymplectic when restricted to~$\Man$.
\end{corollary}

\subsection{Relation between $\protect\Ham{H}$ and $\Mp$}

As a subset of $\sympman$, the hidden constraint set $\Mp$ is given by the points on $\Man$ where the Hamiltonian vector field $X_H$ is tangential to $\Man$.
Let $\Proj$ denote the projection onto~$\Mp$ defined in \autoref{cor:projection_onto_MH}.
\begin{proposition}
\label{prop:MH_extrinsic}
Under \autoref{assump:coisotropy}, the hidden constraint set $\Mp$ is
\begin{equation*}
	\Mp = \big\{\, z \in \Man :  X_H(z)\in\Tan_z\Man \,\big\} .
\end{equation*}
Moreover, the differential equation~\eqref{eq:diffeq_on_Mp} on $\Mp$ can be written
\begin{equation*}
\dot{z} = \Tan_z\Proj\cdot \Ham{H}(z)
.
\end{equation*}
\end{proposition}
\begin{proof}
\begin{enumerate}
\item
If $z \in \Mp$, then by definition there exists $Y \in \Tan_z\Man$ such that
\begin{equation*}
\bracket{\dd H}{X} = ω(Y,X) \qquad \forall X \in \Tan_z \Man
.
\end{equation*}
Since $ω(\Ham{H}) = \dd H$, that is equivalent to
\begin{equation*}
\Ham{H}-Y \in \Tan\Man\sorth
.
\end{equation*}
Noticing that \autoref{assump:coisotropy} means that $\Tan\Man\sorth \subset \Tan\Man$, and using $Y\in \Tan\Man$ yields $\Ham{H} \in \Tan_z\Man$.
\item
The differential equation on $\Mp$ is such that
\begin{equation*}
ω(\dot{z},X) = \bracket{\dd H}{X} \qquad \forall X \in \Tan_z \Man
\end{equation*}
so we obtain
\begin{equation*}
\dot{z} - Y \in \Tan_z\Man\sorth
,
\end{equation*}
and $\dot{z} = \Tan_z\Proj\cdot\Ham{H}(z)$.
\end{enumerate}
\end{proof}

\begin{remark}
\label{rk_Mp_alter}
There are now several ways to compute $\Mp$.
First, without any assumption, one can use the definition \eqref{eq_def_Mp}, and its immediate consequence \autoref{prop_Mp_critical}.
Under \autoref{assump:coisotropy}, one can also use \autoref{prop:MH_extrinsic}.
If the constraint manifold $\Man$ is defined as in \eqref{eqdefMG}, a further useful description of $\Mp$ is
\begin{equation*}
\Mp = \big\{\, z\in \Man : \Poisson{g_i}{H}(z) = 0 \quad i=1,\ldots,\ncons  \,\big\}
.
\end{equation*}
This follows from the observation that 
\begin{equation}
\Ham{H} \in \Tan_z \Man \iff \bracket{\dd g_i}{\Ham{H}} = 0,\quad i=1,\ldots,\ncons
\end{equation}
and $\bracket{\dd g_i}{\Ham{H}} = \Poisson{g_i}{H}$.
\end{remark}

Based on \autoref{thm:MH_manifold}, \autoref{prop_coisotropy} and \autoref{prop_orbit}, we can say much more on the behaviour of $\Ham{H}$ in a neighbourhood of $\Mp$.
Indeed, we have the following result, which is a key ingredient in the well-posedness of \Shake and \Rattle, as will be explained in~\autoref{sec:geometry_of_shake_and_rattle}.

\begin{lemma}
\label{lma_coisotropy}
Let $y \in \Mp$ and define the function
\begin{equation*}
	\begin{split}
		F \colon \fib{y} &\longrightarrow \cons^* \\
		x &\longmapsto \dd\mom(x) \cdot X_H(x)
		.
	\end{split}
\end{equation*}
Then, under \autoref{assump:non_degeneracy} and \autoref{assump:coisotropy}, the differential of~$F$ at~$y$, i.e., the linear mapping
\[
	\dd F(y):\Tan_y\fib{y}\longrightarrow \cons^* ,
\]
is invertible.
\end{lemma}

\begin{proof}
	In terms of the previously introduced basis $\{e_i\}_{i=1,\ldots,\ncons}$, the function~$F$ is given by
	\[
		F(x) = \sum_{i=1}^\ncons \bracket{\dd g_i(x)}{\Ham{H}(x)} e_i
			= -\sum_{i=1}^\ncons \bracket{\dd H(x)}{\Ham{g_i}(x)} e_i
	\]
	Under \autoref{assump:coisotropy}, it follows from \autoref{prop_coisotropy} and \autoref{prop_orbit} that $\Ham{g_1}(y),\ldots,\Ham{g_\ncons}(y)$ is a basis for $\Tan_y\fib{y}$.
	Relative to this basis, and the basis $\{e_i\}_{i=1,\ldots,\ncons}$ of $\cons$, the Jacobian matrix of $\dd F(y)$ is given by
	\[
		m_{ij} \coloneqq \bracket{\dd \bracket{\dd H(y)}{\Ham{g_i}(y)}}{\Ham{g_j}(y)} = \Poisson{\Poisson{H}{g_i}}{g_j}(y)
	\]
	Define $\rho_i(x)\coloneqq\bracket{\dd H(x)}{\Ham{g_i}(x)}$. 
	Then $m_{ij} = \bracket{\dd\rho_i(y)}{\Ham{g_j}(y)}$.
	Since $y\in\Mp$ we have $\rho_i(y)=0$.
	Now, under \autoref{assump:non_degeneracy} and the exact same argument as in the proof of \autoref{thm:MH_manifold}, it follows that $m_{ij}$ is invertible. 
	This concludes the proof.
\end{proof}



\section{Geometry of \protect\Shake and \protect\Rattle} 
\label{sec:geometry_of_shake_and_rattle}


\begin{figure}
	\centering
	\includegraphics{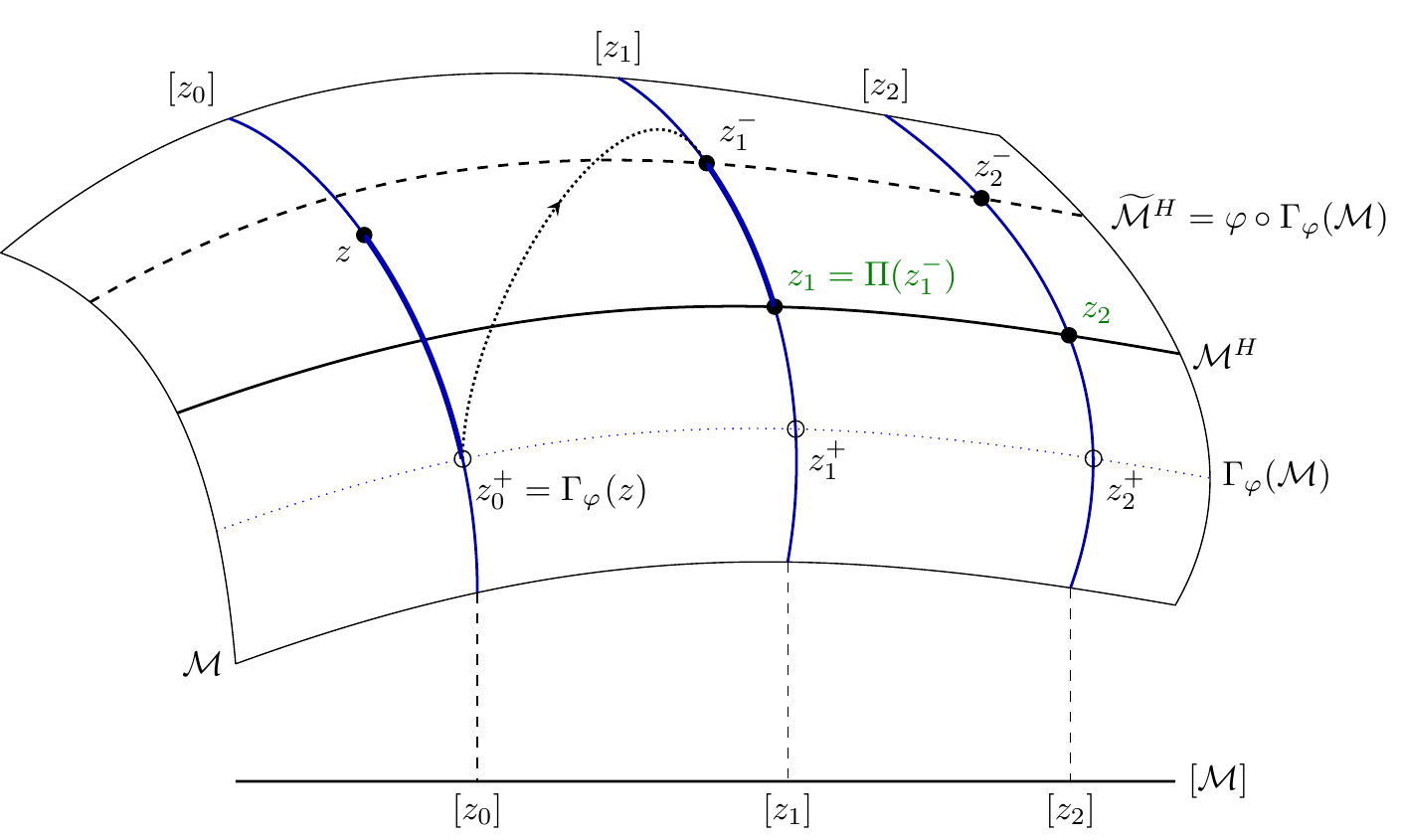}
	\caption[Manifold]{
	An illustration of the \Shake and \Rattle method with underlying symplectic method $\dflow[]$.
	See also \autoref{fig_pendulum} for an equivalent illustration in the familiar case of the pendulum.
	The first point $z$ must lie in the vicinity of $\Mp$. 
	The subsequent points $z_1^-,z_2^-,\ldots$ produced by \Shake will all lie in the modified hidden constraint manifold $\Mpd$.
	The points $z_1,z_2,\ldots$ produced by \Rattle will lie on $\Mp$.
	Notice how the points $z_k^-$, $z_k$ and $z_k^+$ always lie on the same fibre.
	Moreover, the precise location of a point on the fibre (for instance $z$ on the figure) is irrelevant for both \Shake and \Rattle.
	}
	\label{fig_shake}
\end{figure}

Geometrically, the basic principle of \Shake, defined in \autoref{def:geometric_shake}, with $\dflowh$ as an underlying method, can be described as follows.
For some initial data $z\in\Man$, slide along the fibre with $z^+\in\fib{z}$ such that $z^- = \dflowh(z^+)$ ``lands'' again on the submanifold~$\Man$.
The \Rattle method, defined in \autoref{def:geometric_rattle}, is then a post-processed version of \Shake, which is described geometrically as follows.
For some initial data $z_0\in\Mp$, take one step with \Shake landing on $z_1^- \in\Man$, then slide along the fibre $\fib{z_1^-}$ to end up on~$z_1\in\Mp\cap\fib{z_1^-}$.
This process is visualised in \autoref{fig_shake}.

If we assume that the \Shake map $\Sh\colon O\to O$ is well defined for some open subset $O\subset\Man$ containing~$\Mp$, then we may define a sliding map~$\launch{\dflowh} = \dflowh^{-1}\circ \Sh$.
Consequently, it follows from \autoref{def:geometric_shake} that the sliding map $O\ni z\mapsto \launch{\dflowh}(z) \in O$ is defined implicitly by the equation
\begin{equation}\label{eq:slidingdef}
	\dflowh(z^+)\in\Man, \quad z^+\in\fib{z}\cap O, \quad \launch{\dflowh}(z) = z^+ .
\end{equation}
Notice that: (i) $\launch{\dflowh}$ is fibre preserving, i.e., trivially presymplectic, and (ii) $\launch{\dflowh}$ is a projection, i.e., $\launch{\dflowh}\circ\launch{\dflowh}=\launch{\dflowh}$. 
Since $\Sh = \dflowh\circ\launch{\dflowh}$ it follows that \Shake, if it is well-defined, is a \emph{fibre mapping}, i.e., it maps fibres to fibres.
It is, in fact, a little bit more than that, since it maps the whole fibre $\fib{z}\cap O$ to the same point~$\Sh(z)$.
Hence, when using \Shake it is not important where on the initial fibre $\fib{z}$ one starts (as long as it is close enough to $\Mp$ so that \Shake is well-defined).
Furthermore, regardless of where on the fibre one starts, after one step \Shake remains on the \emph{modified} hidden constraint set, given by
\begin{equation*}
	\Mpd \coloneqq \dflowh\circ \launch{\dflowh}(O).
\end{equation*}
Since $\launch{\dflowh}$ is a projection, $\Mpd$ is strictly smaller than~$O$. 
If \Shake is well-defined,~$\Mpd$ is in fact a symplectic submanifold of~$\Man$ (see \autoref{prop_Mpd_symplectic} below).

Let $\Proj$ be the projection on $\Mp$ given in \autoref{cor:projection_onto_MH}.
Assume that $\Proj$ is well-defined on~$O$.
Then \Rattle is given by $\Rh = \Proj\circ \Sh$.
Notice that \Shake and \Rattle define exactly the same fibre mapping.
In particular,
\begin{equation}
\label{eq:shakeequalrattle}
\Proj \circ (\Sh)^k  = \Proj\circ (\Rh)^k
.
\end{equation}
Thus, \Rattle is only a cosmetic improvement of \Shake, and has no influence on the numerical scheme except at the last step.

We now give explicit conditions under which \Shake and \Rattle are well defined and can be computed.
More precisely:
\begin{enumerate}
	\item When is a method $\dflowh$ such that the corresponding \Shake and \Rattle methods are well defined?
	\item How can we parameterise $\launch{\dflowh}$ (so that $\Sh$ is computable)?
	\item Will \Shake and \Rattle converge to the solution of equation~\eqref{eq:constrained_system} as $h\to 0$~?
	\item Are $\Sh$ and $\Rh$ presymplectic as mappings $O\to O$ ?
	\item Is $\Rh$ symplectic as a mapping $\Mp \to \Mp$, and $\Sh$ symplectic as a mapping $\Mpd\to\Mpd$ ?
	\item Can $\Sh$ or $\Rh$ be reversible?
\end{enumerate}
These questions are addressed in the remainder of this section.

\subsection{Well-Posedness}

In order for \Shake and \Rattle to be well defined, we need the ``sliding process'' to have a locally unique solution.
Whether so or not depends on the map $\dflowh$. 


\begin{theorem}
\label{thm:Mvariational_methods}
	Suppose that \autoref{assump:non_degeneracy} and \autoref{assump:coisotropy} hold.
	Consider a (smooth) method~$\dflowh$, consistent with $\dot z = X_H(z)$.
Then for $h$ small enough and for $z\in\Man$ in a neighbourhood of $\Mp$ the equation
\begin{equation}
\label{eq:launch}
		\mom \big(\dflow(z^+)\big) = 0, \quad z^+\in \fib{z}
\end{equation}
	has a unique solution.
\end{theorem}

\newcommand*\momcomp[1][]{\mom\big(#1\big)}

\begin{proof}
	Let $z\in\Mp$. We define the function $F_h \colon \fib{z}\to\cons^*$ by
	\[
		F_h(z^+) \coloneqq \int_0^1 \ddp{}{h}\Big(\mom\big(\dflow[h](z^+)\big)\Big)(h \tau) \dd\tau
		.
	\]
	We see that~$F_h$ depends smoothly on~$h$.
	Notice that for $h \neq 0$, 
	\[
		F_h(z^+) = \frac{\mom(\dflow(z^+)) - \mom(\dflow[0](z^+))}{h} = \frac{\mom(\dflowh(z^+))}{h}
		,
	\]
	because $\dflow[0](z^+) = z^+$ and $z^+ \in \Man \iff \mom(z^+) = 0$. 
	
	Consider now the case $h=0$.
	The method $\dflowh$ is consistent, so $(\dd\dflowh/\dd h)|_{h=0}(z^+) = X_H(z^+)$, which shows that
	\begin{equation}
	\label{eq_Fzero}
		F_0(z^+) =  \dd\mom(z^+)\cdot X_H(z^+) .	
	\end{equation}
	
	We want to find a neighbourhood $O_h\subset\fib{z}$ of $z$ such that the equation $F_h(z^+)=0$ with $z^+\in O_h$ has a unique solution for small enough~$h$. 
	This will prove the claim.
	The strategy is to show that $\dd F_h(z)\colon \Tan_z\fib{z}\to\cons^*$ is non-singular.
	We start with the case~$h=0$. 
	
	Using \eqref{eq_Fzero} and appealing to \autoref{lma_coisotropy}, $\dd F_0(z)$ is invertible under \autoref{assump:non_degeneracy} and \autoref{assump:coisotropy}.
	Thus, by the inverse function theorem, we can find an open neighbourhood $O_0\subset \fib{z}$ such that $F_0 \colon O_0\to  F_0(O_0)$ is a diffeomorphism. 
	Also, since $F_h$ depends smoothly on~$h$, it follows that~$\dd F_h(z)$ is invertible for small enough~$h$. 
	Thus, for small enough~$h$, we can find an open neighbourhood $O_h\subset \fib{z}$ such that $F_h \colon O_h\to  F_h(O_h)$ is a diffeomorphism. 
	
	It remains now to show that $0\in  F_h(O_h)$, so that the equation $F_h(z^+)=0$ with $z^+\in O_h$ has a unique solution.
	Now, if $z\in\Mp$ it follows from~\autoref{prop:MH_extrinsic} that $X_H(z)$ is tangential to $\Man$, which means that $\dd \mom(z) \cdot X_H(z)=0$, so we get $F_0(z)=0$. 
	Thus, $0\in F_0(O_0)$, and it follows by smoothness that $0\in F_h(O_h)$ for small enough~$h$.
	%
\end{proof}

\begin{remark}
The coisotropy assumption is essential for the result of \autoref{thm:Mvariational_methods}, because it uses \autoref{lma_coisotropy} which depends on that assumption in an essential manner.
On the other hand, if \eqref{eq:launch} has a unique solution in a neighbourhood of $\Mp$, then \autoref{assump:coisotropy} must hold.

Indeed,  the theorem implies that $\fib{z}$ is locally diffeomorphic to $\cons^*$.
This implies that their dimension is the same, so $\ker \omegaM$ must have the same dimension as $\cons^*$.

As a result,
\begin{equation*}
\dim \ker \omegaM = \dim \cons^* = \codim \Man = \dim \Tan \Man\sorth
.
\end{equation*}
Now, in general $\ker \omegaM = \Tan\Man \cap \Tan\Man\sorth$, so we get  $\Tan\Man\sorth \subset \Tan\Man$, which implies that $\Man$ is a coisotropic submanifold.
\end{remark}

The result of \autoref{thm:Mvariational_methods} allows one to define the sliding map $\launch{\dflow}$ by \eqref{eq:slidingdef}.
The corresponding \Shake and \Rattle maps, $\Sh$ and $\Rh$, are thus defined when $h$ is sufficiently small.


\subsection{Fibre Parametrisation}

When the manifold $\Man$ is defined implicitly as the locus of functions $g_i$, one can express the effect of $\launch{\dflow}$ using the flows of the Hamiltonian vector fields $\Ham{g_i}$.

\begin{proposition}\label{prop:parametrisation_of_launch}
	Suppose that the $\dflowh$-\Shake method is well defined, so that $\launch{\dflowh}$ is well defined.
	Under \autoref{assump:coisotropy} there exist functions $\lambda_1,\ldots,\lambda_\ncons\in\mathcal{C}^\infty(\Man)$ such that the sliding map $\launch{\dflowh}$ is given by
	\begin{equation*}
		\launch{\dflowh}(x) = \exp\Big(h\sum_{i=1}^\ncons\lambda_i(x)\Ham{g_i}\Big)(x)
	\end{equation*}
\end{proposition}

\begin{proof}
	Follows directly from \autoref{cor:fibre_parametrisation}.
\end{proof}

\subsection{Convergence}

We may now show the convergence of \Shake and \Rattle.
The proof is essentially a standard convergence argument for \Rattle, which is a numerical method on the manifold~$\Mp$.
The convergence of \Shake is then obtained using \eqref{eq:shakeequalrattle}.

\begin{proposition}\label{pro:convergence}
	Suppose that \autoref{assump:non_degeneracy} and \autoref{assump:coisotropy} hold. Let~$\dflowh$ be a method consistent with $\dot z = X_H(z)$.
	Then the $\dflowh$--\Shake and $\dflowh$--\Rattle methods are convergent of order at least~$1$.
\end{proposition}

\begin{proof}
	We first turn to \Rattle.
	The continuous system is an ordinary differential equation on $\Mp$ with vector field $\Tan\Proj\circ\Ham{H}$ (see \autoref{prop:MH_extrinsic}).
	Since it is an ordinary differential equation, we only need to show that \Rattle (defined in \autoref{def:geometric_rattle}) is consistent and standard arguments may then be used to show convergence of order~1 (see \cite[\S\!~II.3]{HaNoWa93}).
	Using that $\Rh = \Proj\circ\dflowh\circ\launch{\dflowh}$ and differentiating at $h=0$ we get
	\begin{equation*}
		\ddp{\Rh}{h}\Big|_{h=0} = \Tan\Proj \circ\Big(\ddp{\dflowh}{h}\Big|_{h=0} + \ddp{\launch{\dflowh}}{h}\Big|_{h=0}\Big)
		,
	\end{equation*}
	which follows since $\dflow[0]=\mathrm{id}$ and $\launch{\dflow[0]}=\mathrm{id}$.
	The second term is in the kernel of $\Tan\Proj$, so
	\begin{equation*}
	\ddp{\Rh}{h}\Big|_{h=0} = \Tan\Proj\circ\ddp{\dflowh}{h}\Big|_{h=0}
	.
	\end{equation*}
	The assumption that $\dflowh$ is consistent with $\dot z = \Ham{H}(z)$ means that $\ddp{\dflowh}{h}\big|_{h=0}=\Ham{H}$.
	Consistency of $\Rh$ then follows.

	Now, using \eqref{eq:shakeequalrattle} and that $\launch{\dflow}(z)= z + \mathcal{O}(h)$ when $h\to 0$, we also see that \Shake converges of order at least~1.
\end{proof}

\begin{remark}
Interestingly, the local error of \Shake is only~$\mathcal{O}(h)$, so standard arguments do not apply to show convergence.
However, due to the fact that \Shake is the same fibre mapping as \Rattle, this error does not accumulate, and the global error is still~$\mathcal{O}(h)$.
\end{remark}

\subsection{Symplecticity}

We examine in which sense \Shake and \Rattle may be regarded as symplectic methods.
The essential result is that both \Shake and \Rattle are \emph{presymplectic}, i.e., they preserve the presymplectic structure of $\Man$.

\begin{proposition}\label{pro:presymplecticy_of_shake}
	Let $\dflowh$ be a symplectic method.
	Then the corresponding \Shake map $\Sh$ and \Rattle map $\Rh$, regarded as mappings $\Man\to \Man$, are presymplectic.
\end{proposition}

\begin{proof}
	Since $\launch{\dflowh}$ is trivially presymplectic, it is in particular presymplectic, so
	\[
		\omegaM(u,v) = \omegaM(\Tan\launch{\dflowh}\cdot u,\Tan\launch{\dflowh}\cdot v),
		\qquad \forall\; u,v\in\Tan\Man.
	\]
	Further, since $\dflowh$ is symplectic it follows that
	\[
		\begin{split}
			\omegaM(u,v) &= \omegaM\big(\Tan\dflowh\cdot(\Tan\launch{\dflowh}\cdot u),\Tan\dflowh\cdot(\Tan\launch{\dflowh}\cdot v)\big)
			\\ &= \omegaM\big(\Tan(\dflowh\circ\launch{\dflowh}) \cdot u,\Tan(\dflowh\circ\launch{\dflowh}) \cdot v\big)
			\\ &= \omegaM\big(\Tan \Sh\cdot u,\Tan \Sh \cdot v\big), \qquad \forall\; u,v\in\Tan\Man.
		\end{split}
	\]
	Thus, $\Sh$ is presymplectic.

	Moreover, since $\Rh = \Proj \circ \Sh$ and $\Proj$ is trivially presymplectic, $\Rh$ is also presymplectic.
\end{proof}

\begin{proposition}
\label{prop_Mpd_symplectic}
	Let $\dflowh$ be a symplectic method.
	Under \autoref{assump:non_degeneracy} and \autoref{assump:coisotropy}, the set $\Mpd$ is a symplectic submanifold of $\Man$, with symplectic form $\omegaMpd = \incl{\Mpd}^*\omegaM$.
\end{proposition}
\begin{proof}
	Under \autoref{assump:non_degeneracy} it follows from \autoref{thm:MH_manifold} that the set $(\Mp,\omegaMp)$ is a symplectic submanifold of~$\Man$.
	We first show that $\Mpd$ is diffeomorphic to $\Mpd$, and thus also a submanifold of~$\Man$, by constructing a diffeomorphism $\Mp\to\Mpd$.
	Under \autoref{assump:non_degeneracy} and \autoref{assump:coisotropy}, the map $\launch{\dflowh}\colon \Mp\to\Man$ is well defined.
	Let $z\in\Mp$.
	By construction we have $\ker(\Tan_z\launch{\dflowh}) = \Tan_z\fib{z}$.
	Thus, $\Tan_z\launch{\dflowh}\colon \Tan_z\Mp\to\Tan_{\launch{\dflowh}(z)}\Man$ is injective, so $\launch{\dflowh}\colon \Mp\to \launch{\dflowh}(\Mp)$ is a diffeomorphism.
	Next, since $\dflowh\colon \sympman\to\sympman$ is a diffeomorphism, it also holds that $\dflowh\colon \launch{\dflowh}(\Mp)\to \dflowh(\launch{\dflowh}(\Mp))=\Mpd$ is a diffeomorphism.
	Thus, $\Sh=\dflowh\circ\launch{\dflowh}$ is a diffeomorphism as a map $\Mp\to\Mpd$, so $\Mp$ and $\Mpd$ are diffeomorphic.

	Next, we show that the form $\omegaMpd=\incl{\Mpd}^*\omegaM$ is nondegenerate.
	Let $y = \Sh(z)$ and $u,v\in\Tan_{y}\Mpd$.
	From \autoref{pro:presymplecticy_of_shake} it follows that $\Sh$ is presymplectic, so
	\[
		\omegaM(u,v) = \omegaM(\Tan_{y}\Sh^{-1}\cdot u,\Tan_{y}\Sh^{-1}\cdot v) = \omegaMp(\Tan_{y}\Sh^{-1}\cdot u,\Tan_{y}\Sh^{-1}\cdot v) .
	\]
	Since $\Tan_{y}\Sh^{-1}\colon \Mpd\to\Mp$ is invertible, and since $(\Mpd,\omegaMp)$ is a symplectic submanifold, it follows that $\omegaM(u,v)=0$ for all $v\in\Tan_{y}\Mpd$ only if $u=0$, which shows that $(\Mpd,\omegaMpd)$ is a symplectic manifold.
\end{proof}

The following result follows directly from \autoref{thm:MH_manifold}, \autoref{pro:presymplecticy_of_shake} and \autoref{prop_Mpd_symplectic}. 

\begin{corollary}\label{cor:symplecticity_of_shake}
	Let $\dflowh$ be a symplectic method and
	assume that \autoref{assump:non_degeneracy} and \autoref{assump:coisotropy} hold.
	Then:
	\begin{itemize}
		\item 	The \Shake map $\Sh$, regarded a mapping $\Mpd\to\Mpd$, is symplectic.
		\item The \Rattle map $\Rh$, regarded as a mapping $\Mp\to\Mp$, is symplectic.
	\end{itemize}
\end{corollary}


\subsection{Time Reversibility}

	Just as in the holonomic case, if the underlying method $\dflow$ is \emph{symmetric}, i.e., $\dflow^{-1} = \dflow[-h]$, then \Rattle is also symmetric and thus of second order.
	Note that \Shake can never be symmetric, because although $\Sh$ preserves $\Mpd$, the reverse method $\Smap{-h}$ does not.


\begin{proposition}
	If the underlying method $\dflow$ is symmetric, then so is \Rattle, considered as a method from $\Mp$ to $\Mp$.
\end{proposition}
\begin{proof}
	This follows from the symmetry property of the map $\launch{\dflow}$.
	In general, if $z_0^+ = \launch{\dflow[]}(z_0)$, $z_1^- =  \dflow[](z_0^+)$, and $z_1 = \Proj (z_1^-)$ (see \autoref{fig_shake}), then
	\begin{equation}
	\label{eq:launch_symmetric}
	z_1^- = \launch{\dflow[]\inv}(z_1)
	.
	\end{equation}
	This follows from \autoref{thm:Mvariational_methods}, because $z_1^-$ is a solution of the equation $\mom\big(\dflow[]\inv(z_1)\big) = 0$, so it must be equal to $\launch{\dflow[]^{-1}}(z_1)$.

	Suppose that we start from a point $z_0 \in \Mp$.
	The image by the \Rattle map is $z_1 = \Rh(z_0)$.
	Now, since $z_0 = \Proj(z_0^+)$ (as we assumed that $z_0 \in \Mp$), and \eqref{eq:launch_symmetric}, we obtain $z_0 = \Proj \circ \dflow\inv \circ \launch{\dflow\inv}$.
	
	If we assume now that $\dflow$ is symmetric, i.e., $\dflow[h]\inv = \dflow[-h]$, we obtain $z_0 = \Rmap{-h}(z_1)$, so \Rattle is symmetric.
\end{proof}


\section{Examples} 
\label{sec:examples}

In this section we give examples of constrained problems that can be solved with the geometric \Shake and \Rattle methods.
In \autoref{sec_example_high_index} we also give non-trivial examples where the nondegeneracy assumption fails to hold.

\subsection{Holonomic Case} 
\label{sec:holonomic_case}

In this section we study the classical, so-called \emph{holonomic} case, where the constraints depend only on the position $q$, and not on the momentum $p$.

\subsubsection{Classical Assumptions}
\label{sec:holonomic_solvability}

Consider the classical case of constrained mechanical systems, where the symplectic manifold $\sympman$ is simply $\RR^{2n}$, with coordinates $(q^i,p_i)$, equipped with the canonical symplectic form $ω = ∑_i \dd q^i \wedge \dd p_i$.
The constraints are given by  $g(q) = 0$ for a function $g$ defined from $\sympman$ to $\cons \coloneqq \RR^m$.
It should be emphasised that $g$ does not depend on $p$, i.e., $g_p = 0$.
The fibre passing through the point $z = (q,p)$ is then an affine subspace described parametrically by
\begin{equation*}
\big\{\, (q,p+ g_q(q)\trans λ) : λ \in \RR^m \,\big\}
.
\end{equation*}

\autoref{assump:coisotropy} is automatically fulfilled, because $\ddp{g_i}{p_j} = 0$ implies
\begin{equation*}
\Poisson{g_i}{g_j} = 0
.
\end{equation*}

Moreover, the standard assumption placed on the Hamiltonian $H$ is that the matrix
\begin{equation}
\label{eq_H_Solvability_HLW}
g_q(q) H_{pp}(q,p) g_q(q)\trans \qquad \text{be invertible}
,
\end{equation}
(see \cite[\S\!~VII.1.13]{HaLuWa06}, \cite{Re96}), which implies \autoref{assump:non_degeneracy}.

The condition \eqref{eq_H_Solvability_HLW} is more stringent than \autoref{assump:non_degeneracy}, since the latter only involves critical points of $H$ on a fibre, i.e., on the points lying on the hidden constraint manifold $\Mp$.
It is clear that the condition \eqref{eq_H_Solvability_HLW} is too restrictive when the fibres are compact, and we will examine such an example in \autoref{sec_example_hopf}.

\subsubsection{Integrators on cotangent bundles}

We explain here how to construct symplectic integrators on cotangent bundles $\Tan^*\ManQ$, for a configuration manifold 
\begin{equation}
\ManQ \coloneqq \big\{\, q \in \RR^n : g(q) = 0 \,\big\}
.
\end{equation}
The constraint manifold $\Man$ is then
\begin{equation}
\Man \coloneqq \big\{\, (q,p) \in \Tan^*\RR^{n} : g(q) = 0 \,\big\}
.
\end{equation}
The crucial observation is that $\Tan^*\ManQ$ is \emph{canonically} symplectomorphic to $\Mq$.
We can now use \Shake to construct a symplectic integrator on $\Tan^*\ManQ$ as follows.

Suppose that we are given a Hamiltonian $H$ on $\Tan^*\R^n$.
Starting from a point $\xi_0 \in \Tan^*\ManQ \simeq \Mq$, one lifts it to a point $z_0 \in \Man$ such that $\fib{z_0} = \xi_0$.
\Shake then produces a new point $z_1^-$, which by projection gives $\xi_1 \coloneqq [z_1^-]$.
Note that $\xi_1$ does not depend on the choice of the point $z_0$ on the fibre.
In particular, note that the mapping $\xi_0 \mapsto \xi_1$ is exactly the same if \Rattle is used instead of \Shake.
The mapping $\xi_0 \mapsto \xi_1$ is then a symplectic integrator on $\Tan^*\ManQ$ for the Hamiltonian on $\Mq$ given by the projection of $\incl{\Mp}^{*}H$.
This method has the same order of convergence as \Rattle.

For instance, in \autoref{ex:pendulum} the configuration manifold $\ManQ$ is a circle, so $\Man$ is the vector bundle of planes above the points of the circle.
Each such plane is spanned by two vectors: one normal and one tangential to the circle.
\Shake and \Rattle differ only by the normal component.
Since a co-vector in $\Tan^*\ManQ$ is defined by its scalar product with tangent vectors in~$\Tan\ManQ$, its normal component has no effect, so \Shake and \Rattle have identical effect as integrators on $\Tan^*\ManQ$.

\subsubsection{Classical \Shake and \Rattle}
\label{sec_classical}

The classical \Shake method is only defined for separable Hamiltonian functions (see \cite[\S\!~VII.1.23]{HaLuWa06}, \cite{Re96},  \cite{LeSk1994}).
However, it is readily extended to general Hamiltonians as the mapping $(q_0,p_0) \mapsto (q_1,p^-_1)$ defined by
\begin{equation*}
\begin{split}
p^+_0 &= p_0 - \frac{h}{2}g_q(q_0)\trans λ \\
p_{1/2} &= p^+_0 - \frac{h}{2} H_q(q_0,p_{1/2})  \\
q_{1} &= q_{0} + \frac{h}{2} \big(H_p(q_0,p_{1/2}) + H_p(q_{1},p_{1/2})\big) \\
p^-_{1} &= p_{1/2} - \frac{h}{2}H_q(q_1,p_{1/2}) \\
0 &=  g(q_1) 
 .
\end{split}
\end{equation*}

The classical \Rattle method is obtained by adding a projection step onto the manifold $\Mp$, i.e., the next step is instead $(q_1,p^-_1)\mapsto (q_1,p_1)$ defined, on top of \Shake, as
\begin{equation*}
\begin{split}
p_{1} &= p^-_{1} + g_q(q_1)\trans μ \\
0 &= g_q(q_1) H_p(q_1,p_1) 
.
\end{split}
\end{equation*}

We see that the mapping $(q_0,p^+_0) \mapsto (q_1,p^-_1)$ is the Störmer-Verlet method, so  the classical \Shake and \Rattle methods  are obtained using the Störmer-Verlet method as the underlying unconstrained symplectic method.

Similarly, in \cite[\S\!~VII.1.3]{HaLuWa06}, a first order method is described as
\begin{equation*}
\begin{split}
p^+_0 &= p_0 - h g_q(q_0)\trans λ \\
p^-_1 &= p^+_0 - h H_q(q_0,p^-_1)  \\
q_1 &= q_0 + h H_p(q_0,p^-_1) \\
0 &= g(q_1) 
\end{split}
\end{equation*}
We see that the mapping $(q_0,p^+_0) \mapsto (q_1,p^-_1)$ is the symplectic Euler method, so this is the \Shake method obtained using symplectic Euler as the underlying unconstrained symplectic method.
When complemented with the projection step 
\begin{equation*}
\begin{split}
p_1 &= p^-_1 -  h g_q(q_1)\trans μ \\
0 &= g_q(q_1)H_p(q_1,p_1) 
,
\end{split}
\end{equation*}
that method becomes the corresponding \Rattle method.


\subsection{Examples of Failing Nondegeneracy Assumption}
\label{sec_example_high_index}

Let us examine an example where \autoref{assump:non_degeneracy} is not fulfilled.
The manifold $\sympman$ is~$\RR^4$ with coordinates $q,p,\qb,\pb$, and the symplectic form is $ω = \dd q \wedge \dd p + \dd \qb \wedge \dd \pb$.
The constraint function $\mom$ is 
\begin{equation*}
\mom(q,p,\qb,\pb) = \qb
.
\end{equation*}
The corresponding fibres are the lines parametrised by $\pb$, i.e., the fibres are the lines of equation $(q,p,\qb) = (q_0,p_0,\qb_0)$.
We choose the Hamiltonian 
\begin{equation*}
H = \frac{p^2}{2} + \pb q 
.
\end{equation*}

The hidden constraint manifold is then $\Mp = \{\, q = 0 \,\}$.
The restriction of $H$ on a fibre is a linear function (because $H$ is linear in $\pb$), so its critical points are degenerate, and \autoref{assump:non_degeneracy} is not fulfilled.
As a result, there are extra constraints which prevent the problem from being well-posed on $\Mp$.
It is readily verified that the system has in fact index five instead of three, the tertiary and quaternary constraints being, respectively, $p = 0$ and $\pb=0$.

If we had assumed instead that $H = \frac{p^2}{2}$, the hidden constraint set would be $\Mp = \Man$, because $H$ is now constant on the fibres (because $H$ does not depend on $\pb$).
In that case, there are infinitely many solutions to the problem at hand.

\subsection{Index 1 Constraints}
\label{sec:indexone}

\newcommand*\mdl[2]{\frac{#1+#2}{2}}

Let $\sympman = \RR^{2n+2k}$ with coordinates $q\in\RR^n$, $p\in\RR^n$, $\alpha\in\RR^k$, $\beta\in\RR^k$, symplectic form $\omega = \dd q \wedge \dd p+ \dd \alpha \wedge \dd \beta$, Hamiltonian $H(q,p,\beta)$, and consider the holonomic constraint $\alpha = 0$. Because the constraints are holonomic, the coisotropy assumption is satisfied (see \autoref{sec:holonomic_case}).
The presymplectic manifold is $\Man  = \RR^{2n+k}$ with coordinates $(q,p,\beta)$ and presymplectic form $\nu =  \dd q\wedge \dd p$.
The hidden constraint submanifold is $\Mp = \big\{\, (q,p,\lambda): H_\beta(q,p,\beta)=0 \,\big\}$, giving the presymplectic system
\begin{equation}
\label{eq:indexone}
\begin{aligned}
\dot q &= H_p(q,p,\beta) \\
		\dot p &= -H_q(q,p,\beta) \\
		0 & = H_\beta(q,p,\beta) \\
\end{aligned}
\end{equation}		
The nondegeneracy assumption is that $H_{\beta\beta}$ is non-singular on $\Mp$. When this holds, the secondary constraint can be solved for $\beta=B(q,p)$.
This is the case of index one constraints considered in \cite{mc-mo-ve-wi}.
In this paper, a symplectic integrator is introduced and studied for the index one constrained Hamiltonian system (\ref{eq:indexone}), consisting of a direct application of a symplectic Runge--Kutta method to the full system (\ref{eq:indexone}).
In the case when~$\dflowh$ is the midpoint rule, we now show that this method is equivalent to the application of \Shake or  \Rattle applied on~$\sympman$.
Indeed, Hamilton's equations on~$\sympman$ are given by
\[
\begin{aligned}
\dot q &= H_p(q,p,\beta) \\
\dot p & = -H_q(q,p,\beta) \\
\dot \alpha &= H_\beta(q,p,\beta) \\
\dot \beta &= 0 \\
\end{aligned}
\]
for which the midpoint discretisation, defining $\dflowh$, is
\begin{equation}
\begin{aligned}
 \frac{q_1 - q_0}{h} &= H_p\Big(\mdl{q_0}{q_1},\mdl{p_0}{p_1},\mdl{\beta_0}{\beta_1}\Big) \\
\frac{p_1 - p_0}{h} &= -H_q\Big(\mdl{q_0}{q_1},\mdl{p_0}{p_1},\mdl{\beta_0}{\beta_1}\Big) \\
\frac{\alpha_1-\alpha_0}{h} &= H_\beta\Big(\mdl{q_0}{q_1},\mdl{p_0}{p_1},\mdl{\beta_0}{\beta_1}\Big) \\
\beta_1 - \beta_0 &= 0 \\
\end{aligned}
\end{equation}
while the constraint flow is given by $\exp(\lambda X_\alpha) \colon (q,p,\alpha,\beta) \mapsto (q,p,\alpha,\beta - \lambda)$.
Initially, the primary constraints are satisfied, i.e. $\alpha_0=0$, and the constraint flow determines $\beta_1$ so that $\alpha_1=0$, i.e., so that 
\[
0 = H_\beta\Big(\mdl{q_0}{q_1},\mdl{p_0}{p_1},\mdl{\beta_0}{\beta_1}\Big)
\]
with solution $\mdl{\beta_0}{\beta_1} = B(\mdl{q_0}{q_1},\mdl{p_0}{p_1})$.
This is exactly the method of \cite{mc-mo-ve-wi} in the case of the midpoint rule. 
The final step of {\Rattle} chooses $\beta$ so that $0 = H_\beta(q_1,p_1,\beta)$, but (as we have seen in the general case), this value is irrelevant for the next step.

\subsection{Harmonic Constraints}
\label{sec_example_hopf}

Consider the symplectic manifold $\sympman = \Tan^*\RR^{2}$ equipped with the canonical symplectic form
\begin{equation*}
\omegaP(q_0,p_0,q_1,p_1) = \dd q_0 \wedge \dd p_0 + \dd q_1 \wedge \dd p_1 
,
\end{equation*}
where $\qv = (q_0,q_1)$ and $\pv = (p_0,p_1)$ are canonical coordinates.
We often make the identification $\Tan^{*}\RR^{2} \simeq \CC^{2}$ by introducing the complex coordinates $z_0 = q_0 + \ii p_0$ and $z_1 = q_1 + \ii p_1$.

Consider now the constraint space $\cons = \RR$ and the single constraint function
\begin{equation}
\label{eq:hopf_constraint}
g(\qv,\pv) \coloneqq \|\qv\|^2 + \|\pv\|^2 - 1
.
\end{equation}
Note that the constraint manifold is given by $\Man = S^3$, i.e., the unit sphere in~$\RR^4$.
It is automatically a coisotropic submanifold since it has codimension one (see \autoref{rk:coisotropy_poisson_commute}).

The flow of $\Ham{g}$ expressed in complex coordinates is given by 
\begin{equation*}
\exp(t \Ham{g})(z_0,z_1) = (\ee^{\ii t}z_0,\ee^{\ii t}z_1)	
.
\end{equation*}
Thus, it follows from \autoref{cor:fibre_parametrisation} that the fibre $\fib{(z_0,z_1)}$ passing through $(z_0,z_1)$ is given by the great circle
\begin{equation*}
	\fib{(z_0,z_1)} = \big\{\, (\ee^{\ii s}z_0,\ee^{\ii s}z_1)\in \CC^{2} : s\in \RR \,\big\}.
\end{equation*}
These fibres are exactly the fibres in the classical Hopf fibration of the 3--sphere in circles.
Recall that the Hopf map is the submersion
\begin{equation}
\label{eq_hopf}
(z_0,z_1) \mapsto \big(2z_0 \bar{z_1}, |z_0|^2 - |z_1|^2\big)
\end{equation}
which projects $\CC^2$ onto $\R^3 \simeq \CC \times \RR$ and maps $3$-spheres onto $2$-spheres.
The quotient set $\Mq$ is thus a manifold which is diffeomorphic to the 2-sphere.

We now investigate the validity of \autoref{assump:non_degeneracy} for two different choices for the Hamiltonian function.

\subsubsection{Hamiltonian without Potential Energy}

%
First consider the Hamiltonian
\begin{equation*}
H(\qv,\pv) = \frac{\|\pv\|^2}{2}
.
\end{equation*}
The governing equations \eqref{eqconslag} are in this case given by
\begin{equation}
\label{eq_hopf_dae}
\begin{split}
\dot{\qv} &= (1 + λ)\pv \\
\dot{\pv} &= -λ\qv \\
1 &= \|\qv\|^2 + \|\pv\|^2 
.
\end{split}
\end{equation}

Let $(\qv,\pv)\in \Man$ and let
\begin{equation*}
	\Hr(\theta) := H\big(\cos(\theta)\qv + \sin(\theta)\pv,-\sin(\theta)\qv + \cos(\theta)\pv\big)
\end{equation*}
so that $\Hr$ is the restriction of $H$ to the fibre containing $(\qv,\pv)$ and parameterised by $\theta \in \RR/\ZZ$.
Notice that $\Hr(0) = H(\qv,\pv)$. 
After trigonometric simplification we obtain
\begin{equation*}
\Hr(θ) = \frac{1}{4}\Big( \|\pv\|^2 \big(1+\cos(2θ)\big) + \|\qv\|^2\big(1 - \cos(2θ)\big) + 2 \qv \cdot \pv \sin(2θ)  \Big)
.
\end{equation*}
From \autoref{prop_Mp_critical} it follows that $(\qv,\pv) \in \Mp$ if and only if it is a critical point of $H$ in the direction of the fibre, i.e., $\Hr'(0) = 0$.
By differentiating, one obtains
\begin{equation}
\label{eq_Hprime}
\Hr'(θ) = \frac{1}{2}\Big( (\|\qv\|^2 - \|\pv\|^2) \sin(2θ) + 2 \qv \cdot \pv \cos(2θ) \Big)
,
\end{equation}
which leads to
\begin{equation}
\label{eq_harmonic_Mp}
\Mp = \big\{\, (\qv, \pv) \in \Man : \qv \cdot \pv = 0 \,\big\}
.
\end{equation}
This could also have been derived by differentiating the constraints in the differential algebraic equation \eqref{eq_hopf_dae}, or by using one of the equivalent formulations of \autoref{rk_Mp_alter}.

The critical point $(\qv,\pv)\in \Mp$ is nondegenerate if and only if $\Hr''(\theta) \neq 0$, which, using~\eqref{eq_Hprime}, means that $\|\qv\|^2 - \|\pv\|^2 \neq 0$.
From this we also see that $(\qv,\pv)$ is degenerate critical point if and only if $\qv\cdot\pv = 0$ and $\|\qv\| = \|\pv\|$, which, again using~\eqref{eq_Hprime}, implies that the Hamiltonian is constant on the whole fibre.
A closer examination shows that there are two such fibres.
By the Hopf map these two fibres correspond to the two antipodal points $(\ii,0)$ and $(-\ii,0)$ on the 2--sphere in $\CC\times\RR$.
We call those two points the \emph{exceptional points} and the corresponding fibres the \emph{exceptional fibres}.
If we remove the points where $\|\qv\| = \|\pv\|$ from the constraint manifold, then $\Mp$ becomes a manifold with two connected components (one cannot get from a point where $\|\qv\| < \|\pv\|$ to a point where $\|\qv\| > \|\pv\|$ without crossing a point where $\|\qv\| = \|\pv\|$).

By differentiating the secondary constraints $\qv\cdot \pv = 0$ with respect to time, and substituting for $\dot\qv$ and $\dot\pv$ using~\eqref{eq_hopf_dae} we obtain
\begin{equation}\label{eq:hopf_lambda}
	\lambda = \frac{\|\pv\|^{2}}{\|\qv\|^{2} - \|\pv\|^{2}}.
\end{equation}
Since both the Hamiltonian and the constraint function are conserved along a solution trajectory, it follows that $\|\qv\|$ and $\|\pv\|$ are constant on each trajectory, so the Lagrange multiplier $\lambda$ is constant.
Using these facts it is straightforward to show that the exact solution to~\eqref{eq_hopf_dae} is given by
\begin{equation}\label{eq:hopf_dea_exact_solution}
	\begin{split}
		\qv(t) &= R\Big( \frac{t}{\alpha - 1/\alpha} \Big)\qv_0 \\
		\pv(t) &= R\Big(\frac{t}{\alpha - 1/\alpha}\Big)\pv_0
	\end{split}
\end{equation}
where $(\qv_0,\pv_0)\in \Mp$ are the initial conditions, $\alpha = \pm \, \|\pv_0\|/\|\qv_0\|$, and
\begin{equation*}
	R(\theta) := \begin{pmatrix} \cos(\theta) & -\sin(\theta) \\ \sin(\theta) & \cos(\theta) \end{pmatrix}.
\end{equation*}
(The sign of $\alpha$ depends on which connected component of the constraint manifold the point $(\qv_0,\pv_0)$ belong to.)
Thus, the exact solutions are given by rotations in the $\qv$ and $\pv$ planes.
These trajectories are great circles of the constraint manifold $\Man = S^3$, so both the fibres and the solution trajectories are given by great circles.
Notice that the solution trajectory passing through $(\qv_0,\pv_0)$ crosses the fibre passing through $(\qv_0,\pv_0)$ twice (see \autoref{fig_Hopf_fibre}).

We now turn to the solution of~\eqref{eq_hopf_dae} by \Shake and \Rattle.
Indeed, \Shake works as follows: for some $(\qv_0,\pv_0)$, move along the fibre so that the length of $(\qv_1^-,\pv_1^-)$ is~1 (to ``land'' again on $\Man$).
Now, $(\qv_0^+,\pv_0^+)$ is related to $(\qv_1^{-},\pv_1^{-})$ by an orthogonal reflection in the hyperplane perpendicular to $\Ham{H}(\qv_0^+,\pv_0^+) = (\pv_0^+,0)$.
One can check that this reflection sends fibres to fibres.
It is also straightforward to check that the reflection leaves solution trajectories invariant.
Therefore, the fibre passing through $(\qv_1^-,\pv_1^-)$ intersects the same solution trajectory as the fibre passing through $(\qv_0,\pv_0)$, which means that \Shake jumps between fibres of the same solution trajectory.
Hence, \Rattle reproduces the exact flow of \eqref{eq_hopf_dae} up to a time reparametrisation. 

Let us gather our findings:
\begin{proposition}
Consider the symplectic manifold $\sympman = \Tan^*\RR^{2}\simeq \CC^{2}$, the constraint submanifold
\begin{equation*}
\Man \coloneqq \Big\{\, (z_0,z_1)\in\sympman : |z_0|^2 + |z_1|^2 = 1 \quad\text{and}\quad (2z_0 \bar{z_1},|z_0|^2 - |z_1|^2) \neq (±\ii,0) \,\Big\}
,
\end{equation*}
and the Hamiltonian $H(z_0,z_1) = (\Im(z_0)^{2} + \Im(z_1)^{2})/2$.
Then:
\begin{enumerate}
	\item Both \autoref{assump:non_degeneracy} and \autoref{assump:coisotropy} hold.
	\item The hidden constraint manifold $\Mp$ is defined by \eqref{eq_harmonic_Mp}. 
	It has two connected components.
	\item The exact solution is given by~\eqref{eq:hopf_dea_exact_solution}.
	\item The solution trajectory and fibre corresponding to some initial data in $\Mp$ intersect twice.
	\item \Shake maps between fibres that intersect the same exact solution trajectory.
	\item Up to time reparametrisation, \Rattle produces the exact solution.
\end{enumerate}
\end{proposition}

We now turn to a slightly modified problem, for which \Rattle does not reproduce the exact trajectory.


\begin{figure}
\centering
\includegraphics{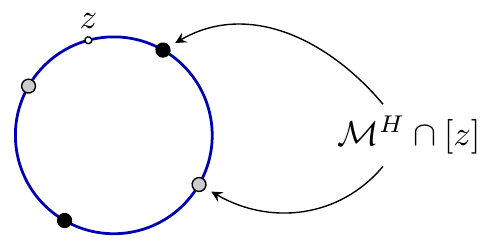}
\caption[Illustration]{An illustration of the intersection of a fibre $\fib{z}$ with $\Mp$ for the system \eqref{eq_hopf_dae}.
Fibres are big circles on the sphere $S^3$.
The two antipodal points of $\Mp$ on each fibre, marked with the same colour on the picture, are part of the same trajectory.
Those two trajectories belong to two distinct connected components of $\Mp$.
}
\label{fig_Hopf_fibre}
\end{figure}


\subsubsection{Hamiltonian with Linear Potential Energy}\label{sub_hopf_lin_potential}

With the same constraint, we consider now instead the Hamiltonian
\begin{equation}
\label{eq_Hamiltonian_grav}
H(\qv,\pv) = \frac{\|\pv\|^2}{2} - \gv \cdot \qv
\end{equation}
where $\gv$ is a fixed given vector.

Now the structure of the hidden constraint set $\Mp$ is more involved.
Let, as above, $\Hr$ be the restriction of $H$ to the fibre passing through $(\qv,\pv)\in \Man$.
Then we get
\begin{equation}
\label{eq_Hprime2}
\Hr'(θ) = \frac{1}{2}\big(\|\qv\|^2 - \|\pv\|^2\big) \sin(2θ) + \qv \cdot \pv \cos(2θ) - \gv\cdot\qv\sin(\theta) + \gv\cdot\pv\cos(\theta) .
\end{equation}
The intersection between $\fib{(\qv,\pv)}$ and $\Mp$ is now more complicated than for the previous Hamiltonian.
We have
\begin{equation*}
	\fib{(\qv,\pv)}\cap\Mp = \big\{\, \big(\qv\cos(\theta)+\pv\sin(\theta),-\sin(\theta)\qv + \cos(\theta)\pv\big): \Hr'(\theta)=0 \,\big\}.
\end{equation*}
From the form~\eqref{eq_Hprime2} of $\Hr'(\theta)$ it follows that $\fib{(\qv,\pv)}\cap\Mp$ contains either: 
\begin{enumerate}[label={(\roman*)}, ref={case~(\roman*)}]
	\item\label{caseregular} four nondegenerate critical points, or
	\item\label{caselower} two nondegenerate critical points, or
	\item\label{casecritical} two nondegenerate critical points and one degenerate critical point.
\end{enumerate}
Which case that occurs depends on the magnitude of the $\cos(2\theta),\sin(2\theta)$ coefficients in~\eqref{eq_Hprime2} in relation to the $\cos(\theta),\sin(\theta)$ coefficients in~\eqref{eq_Hprime2}: if the former are larger we get \ref{caseregular}, if the latter are larger we get \ref{caselower}, in the limit situation we get \ref{casecritical}.
Thus, for some initial data it might (and will, see \autoref{fig_fibres}) happen that the critical points in $\Mp$ consitituting the solution trajectory cease to be nondegenerate, in which case \autoref{assump:non_degeneracy} is not fulfilled and the equation is no longer well posed.

In \autoref{fig_numtests_grav} we give some simulation results using \Shake and \Rattle for globally well posed trajectories.

In \autoref{fig_energyconstraints} we display the energy evolution of \Shake and \Rattle and fulfillment of the hidden constraints for the \Shake method.

In \autoref{fig_fibres} we plot the magnitude of $\Hr'$ for the fibres along a numerical solution trajectory computed with \Shake.
 
\begin{figure}
\centering
\includegraphics{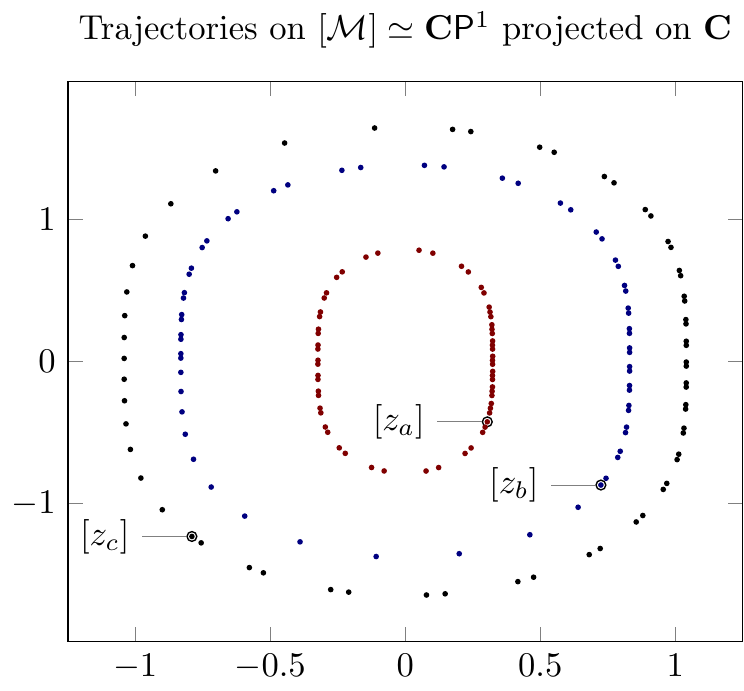}
\caption[Trajectory Plot]{
Simulation results for the Hamiltonian \eqref{eq_Hamiltonian_grav} with the vector $\gv \coloneqq (0,-1)$.
Both the \Shake and \Rattle methods use the midpoint rule as underlying symplectic integrator.
The time step is $h=0.1$ in all the examples.
We choose the initial conditions in \autoref{table_initialconditions}.
The phase trajectories for the three initial conditions $z_a$, $z_b$ and $z_c$ are displayed by first applying the Hopf map \eqref{eq_hopf} and then using a stereographic projection defined as
$\CC\times\RR \ni (\zeta,\rho) \mapsto \frac{\zeta}{\sqrt{|\zeta|^2+|\rho|^2} - \rho} \in \CC$.
The sphere $\Mq$ is thus identified to the Riemann sphere $\CC\mathsf{P}^1$, and the trajectories are plotted on the complex plane.
 
Note that we plot the projected trajectories on $\Mq$, and \Shake and \Rattle become the same mapping when considered as a mapping on $\Mq$.

The essential observation is that the simulated trajectories are closed, thus reproducing an important property of the exact solution.

The innermost trajectory, corresponding to the initial value $z_a$, is also plotted in \autoref{fig_fibres} along with the fibres lying above it.
The energy and the hidden constraint are also plotted for that trajectory in \autoref{fig_energyconstraints}.
}
\label{fig_numtests_grav}
\end{figure}

\begin{table}
\centering
\begin{tabular}{l|lll}
\toprule
 & $z_a$ & $z_b$ & $z_c$ \\
\midrule
$q_0$ & \scriptsize\(\num{-0.7865261200000000}\) & \scriptsize\(\num{-0.4963624948824013}\) & \scriptsize\(\num{0.3477491188213400}\) \\
$q_1$ & \scriptsize\(\num{-0.4043988000000000}\) & \scriptsize\(\num{-0.7319740436366664}\) & \scriptsize\(\num{-0.8131619010029159}\) \\
$p_0$ & \scriptsize\(\num{-0.3880746864163783}\) & \scriptsize\(\num{-0.4275775933953260}\) & \scriptsize\(\num{-0.4368285559113795}\) \\
$p_1$ & \scriptsize\(\num{0.2173391755798215}\) & \scriptsize\(\num{0.1225384882604160}\) & \scriptsize\(\num{-0.0837800227934176}\) \\
\bottomrule
\end{tabular}
\caption[table]{
Three initial conditions used in \autoref{fig_numtests_grav}.
All initial conditions lie on $\Mp$.
}
\label{table_initialconditions}
\end{table}

\begin{figure}
\centering
\subfloat[Energy plot]{\label{sfig_energy}\includegraphics{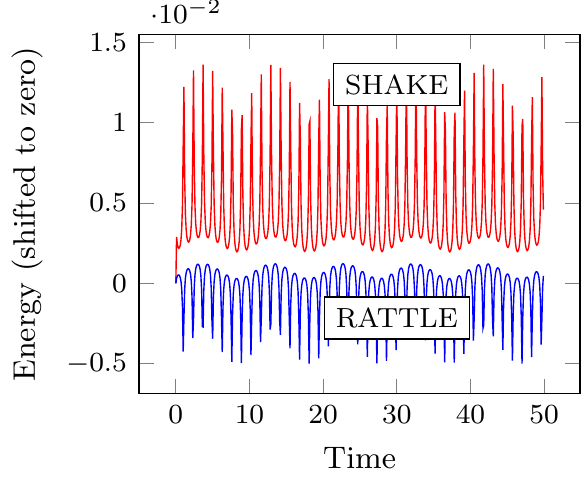}} \quad
\subfloat[Hidden constraint fulfillment for \Shake]{\label{sfig_constraint}\includegraphics{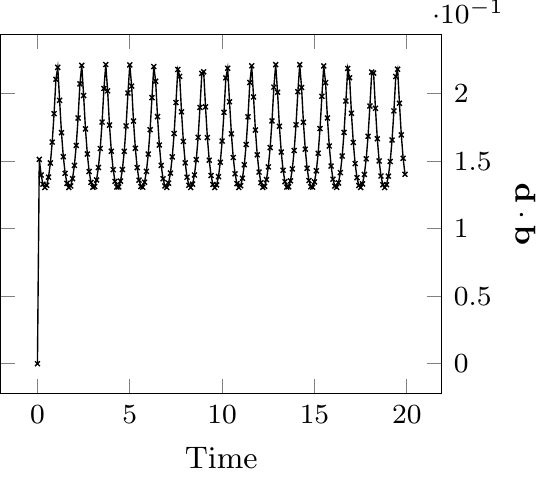}}
\caption[]{
Plot~\protect\subref{sfig_energy} shows the energy evolution for the \Shake and \Rattle methods using the initial condition $z_a$ in \autoref{table_initialconditions}, corresponding to the inner most trajectory in \autoref{fig_numtests_grav}.
Notice that there is a slight difference between \Shake and \Rattle, corresponding to the representation of the method either on $\Mpd$ or $\Mp$ as described above.
Also notice that there is no drift in energy, in accordance with the general theory of symplectic integrators.

Plot~\protect\subref{sfig_constraint} shows the fulfillment of the hidden constraint by the \Shake method using the same initial data~$z_a$.
The hidden constraint stays away from $\Mp$ at the same order of magnitude as the time step ($0.1$ in this case).
A similar plot for \Rattle would show that the hidden constraint is fulfilled up to machine accuracy.
Recall that this does not mean that \Rattle performs better than \Shake, as \Rattle can be seen as a mere optional post-processing step.
Note also that both methods fulfill the primary constraint \eqref{eq:hopf_constraint} up to machine accuracy.
}
\label{fig_energyconstraints} 
\end{figure}

\todo[author=OV,inline]{Add labels on \autoref{fig_fibres} ($\Mp$, $\Mpd$, etc.)?}

\begin{figure}
\centering
\includegraphics{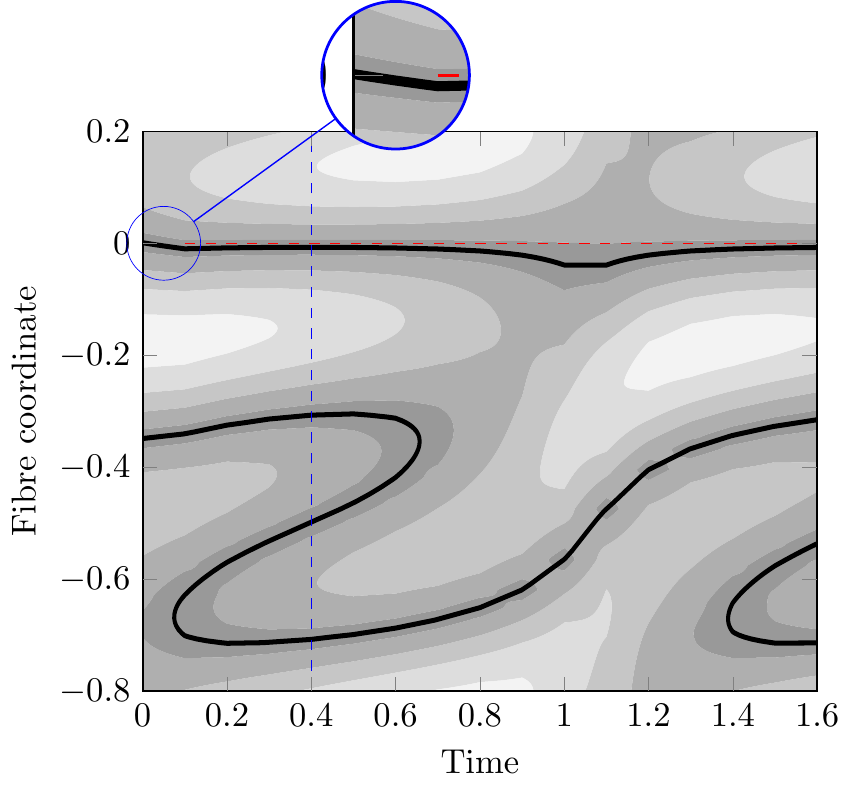}
\caption[Fibre Plot]{
Level curves showing the magnitude of $\Hr'$ (as given by~\eqref{eq_Hprime2}) for the fibres along the trajectory computed using \Shake with initial data $z_a$ given in \autoref{table_initialconditions}.
The points where $\left|\Hr'\right|=0$ are those points lying on $\Mp$.
Black corresponds to~$\left|\Hr'\right|=0$, so the black curves in the graphs give the intersection between~$\Mp$ and the fibres above the trajectory.
By construction, we are taking a coordinate in $[0,1)$ on the circle and shifting it along the computed trajectory so that $\Mpd$ is always at coordinate 0 (the dashed horizontal line in the figure).
The graph is periodic in the vertical direction because the fibres are circles; one such circle given by the dashed vertical line.
The graph is also periodic in the time direction, because the solutions are time periodic (see \autoref{fig_numtests_grav}).

Notice that, although the initial condition lies on $\Mp$, \Shake persistently stays off $\Mp$ after one step (see the zoomed-in section).
This is because $\Mpd$ is slightly offset from $\Mp$, in accordance with the theory in \autoref{sec:geometry_of_shake_and_rattle}.
The result of \Rattle is obtained by projecting vertically onto $\Mp$.

Note that $\Mp$ crosses the fibres either two, three or four times, in accordance with the classification described in \autoref{sub_hopf_lin_potential} (also compare to \autoref{fig_Hopf_fibre}).
In particular, there is another component of $\Mp$ crossing the fibre in the lower part of the figure.
The inflexion points of that curve are points where \autoref{assump:non_degeneracy} is not fulfilled anymore.
The differential algebraic equation is not well defined at those points, and both \Shake and \Rattle fail if such a point is reached.
This failure takes two forms: the Newton iteration may either fail to converge, or it may find a spurious solution by jumping to another component of $\Mp$ which is better behaved.
}
\label{fig_fibres}
\end{figure}

\section{Discussion}

The geometric version of the Dirac constraint algorithm \cite{Go1978} in the instance used here delivers a chain of submanifolds
\[
\Mp \hookrightarrow \Man  \hookrightarrow {\sympman}
\]
in which $\Mp$ is symplectic.
The geometric {\Rattle} algorithm delivers symplectic integrators on $\Mp$. However, ${\sympman}$, the intervening presymplectic manifold $\Man $, its coisotropy, and knowledge of its fibres, are essential to the algorithm.
If the fibres cannot be explicitly parametrised, the algorithm is still formally defined, but more computation would be required in practice---for example, by integrating the fibres numerically to roundoff error.
This is an extreme version of a situation common in numerical analysis, in which allowing a wider class of methods (e.g. implicit Runge--Kutta methods, for which implicit equations have to be solved numerically) enables a wider class of properties. 

If $\Man $ is not coisotropic, then the coisotropic embedding theorem \cite{Go1982} says that there exists a symplectic manifold ${\sympman}'$ such that $\Man $ is coisotropically embedded in ${\sympman}'$ (\autoref{rk:coisotropic_embedding}).
Thus, abstractly at least, one can extend the Hamiltonian on $\Man $ arbitrarily to ${\sympman}'$ and apply the geometric {\Rattle} algorithm, for the rest of the required structure is instrinsic to  $\Man $.
In specific examples it may be possible to carry this out by finding a suitable symplectic vector space ${\sympman}'$.
The same remark holds if the given data is a Hamiltonian on a presymplectic manifold, see \autoref{sec:indexone}.

However there remain many constrained problems which do not fall into the classes considered here. The most
fundamental one has data $({\sympman}, \Man', H)$ where $\Man'$ is a symplectic submanifold of the symplectic manifold ${\sympman}$. We do not know of symplectic integrators for this problem. 
They would provide symplectic integrators for a wide class of symplectic manifolds.
A very general situation is that provided by the geometric version of the Dirac constraint algorithm \cite{Go1978}, which, from presymplectic data $(\Man , \omega, H)$, produces a nested sequence of submanifolds 
\[
\Man_K\hookrightarrow  \Man_{K-1} \hookrightarrow\dots \hookrightarrow \Man_1 := \Man
\]
defined by
\[
\Man_{K+1} := \big\{\, x\in\Man_K\colon \dd H(x)\in \omega(\Man_K) \,\big\}
\]
and a (possibly nonunique) vector field $X$ such that $\incl{\Man_K}(\interior{X}{\omega-\dd H}) = 0$. One would like to integrate an index-$K$ DAE on $\Man$ or an index-$K+1$ DAE on a symplectic embedding of $\Man$ so as to preserve the constraints and $\incl{\Man_K}\omega$.

Finally, we mention another class of integrators for the holonomic case, known as \textsc{spark}, for Symplectic Partitioned Additive Runge--Kutta \cite{Ja1996}.
These generalise {\Rattle} to higher order. 
They are \emph{partitioned} (use different RK coefficients for the $q$ and $p$ components) and \emph{additive} (use different RK coefficients for the constraint and regular forces).
The holonomic assumption is used in two critical steps: first, it means that the flow of the constraint force is given by Euler's method; second, it means that the $q$-component of the constraint forces vanishes. 
This allows their RK coefficients to be arbitrary, which means that the RK coefficients of the $p$-component can be arbitrary, and can be chosen to include stages at the endpoints.
Thus, this approach does not immediately give an algorithm for problems of the type $({\sympman}, \Man', H)$ mentioned above.
The situation is similar to the relationship between splitting methods and RK methods; we do not know if \textsc{spark} can be adapted to more general constraints.

\section*{Acknowledgements}

O.~Verdier would like to acknowledge the support of the \href{http://wiki.math.ntnu.no/genuin}{GeNuIn Project}, funded by the \href{http://www.forskningsradet.no/}{Research Council of Norway}, the Marie Curie International Research Staff Exchange Scheme Fellowship within the \href{http://cordis.europa.eu/fp7/home_en.html}{European Commission's Seventh Framework Programme} as well as the hospitality of the \href{http://ifs.massey.ac.nz/}{Institute for Fundamental Sciences} of Massey University, New Zealand, where some of this research was conducted.
K.~Modin would like to thank the Marsden Fund in New Zealand, the \href{http://www.math.ntnu.no}{Department of Mathematics at NTNU} in Trondheim, the \href{http://www.kva.se}{Royal Swedish Academy of Science} and the \href{http://www.vr.se}{Swedish Research Council}, contract VR-2012-335, for support.
We would like to thank the reviewers for helpful suggestions.

\bibliographystyle{abbrvnat}
\bibliography{ref}
\end{document}